\documentclass[12pt,letterpaper]{amsart}

\usepackage{amsmath,amssymb}
\usepackage{amsthm}
\usepackage{graphicx}
\usepackage{hyperref}
\usepackage{mathtools}
\usepackage{verbatim}
\usepackage[margin=1in]{geometry}

\usepackage{tikz-cd}

\usepackage{enumitem}
\setlist{
leftmargin=2em}

\numberwithin{equation}{section}

\newtheorem{cor}[equation]{Corollary}

\newtheorem{prop}[equation]{Proposition}
\newtheorem{theorem}[equation]{Theorem}
\newtheorem{lemma}[equation]{Lemma}
\newtheorem{proposition}[equation]{Proposition}

\theoremstyle{remark}
\newtheorem{rem}[equation]{Remark}

\newtheorem{remark}[equation]{Remark}

\theoremstyle{definition}

\newcommand{\sm}{\setminus}

\def\<#1>{\left< #1 \right>}
\newcommand{\dd}{\partial}

\newcommand{\C}{\mathbb C}
\newcommand{\D}{\mathbb D}

\newcommand{\Q}{\mathbb Q}
\newcommand{\R}{\mathbb R}

\newcommand{\Z}{\mathbb Z}

\newcommand{\FF}{\mathcal F}

\newcommand{\QQ}{\mathcal Q}
\newcommand{\RR}{\mathcal R}

\newcommand{\chat}{\hat\C}

\DeclareMathOperator{\Dom}{Dom}

\usepackage[textsize=small]{todonotes}

\DeclareMathOperator{\Exp}{Exp}

\usepackage[style=alphabetic, backend=biber]{biblatex}
\title{Non-degenerate near-parabolic renormalization}
\author{Alex Kapiamba}
\date{}

\begin{document}
	\begin{abstract}
		Invariant classes under parabolic and near-parabolic renormalization have proved extremely useful for studying the dynamics of polynomials. The first such class was introduced by Inou-Shishikura \cite{shishikura} to study quadratic polynomials; their argument has been extended to the unicritical cubic case by Yang \cite{Yang} and the general unicritical case by Ch\'eritat \cite{Cheritat22}. However, all of these classes are only applicable to maps which have a fixed point with multiplier close to one, though it is well-known that similar phenomena occur when the multiplier is close to any root of unity. In this paper we define the parabolic and near-parabolic renormalization operators in the general setting and construct invariant classes.
		In the general setting we can observe a new phenomenon: the multiplier may be close to several roots unity. In this case, we show how to directly relate the different near-parabolic renormalizations that arise.
	\end{abstract}
	
	\maketitle

	\section*{Introduction}
	
	Let $f(z)$ be a holomorphic function defined on a neighborhood of a fixed point $z_0\in \mathbb{C}$.
	The \textit{multiplier} of the fixed point is the quantity $\lambda=f'(z_0)$, and $z_0$ is called \textit{parabolic} when $\lambda$ is a root of unity. 
	The parabolic fixed point is called \textit{simple} when $\lambda =1$, and \textit{non-degenerate} when $f''(0)\neq 0$.
	
	When $f(z)$ has a simple non-degenerate parabolic fixed point $z_0$ the local dynamics of $f$ near $z_0$ are relatively tame, however for a holomorphic function $h(z)$ which is a perturbation of $f(z)$ the local dynamics of $f$ can be drastically more complicated. This phenomenon, called \textit{parabolic implosion}, was first studied by  Douady-Hubbard \cite{Orsay1,Orsay2} and Lavaurs \cite{Lavaurs} using Fatou coordinates and Lavaurs maps. While describing the Hausdorff dimension of the Mandelbrot set in \cite{shishikura_boundary}, Shishikura introduced parabolic and near-parabolic renormalization, providing another lens to study parabolic implosion. The parabolic renormalization acts on maps $f$ with a simple non-degenerate parabolic fixed point; Shishikura introduced a class of maps invariant under parabolic renormalization in \cite{shishikura_1}. The near-parabolic renormalization, an example of a \textit{cylinder renormalization} as in \cite{yampolsky}, acts on maps $h$ which are perturbations of maps with simple non-degenerate parabolic fixed points; Inou-Shishikura introduced a class of maps invariant under near-parabolic renormalization in \cite{shishikura}. The Inou-Shishikura class has had several remarkable applications:
	 it is used by Buff-Ch\'eritat to prove the existence of quadratic Julia sets of positive measure \cite{positive_area}, by Cheraghi-Ch\'eritat to partially resolve the Marmi-Moussa-Yoccoz conjecture \cite{MMYconjecture}, and by Cheraghi-Shishikura to make progress towards the celebrated MLC conjecture \cite{CS_satellite}.
	There are numerous other applications of the Inou-Shishikura class, see for example \cite{AC18}, \cite{Cheraghi13}, \cite{Cheraghi17}, \cite{Cheraghi19},  \cite{Shishkura-Yang}. 
	
	All maps in the Inou-Shishikura class have critical points of local degree two, so it is most commonly used to study the dynamics of quadratic polynomials. 
	In \cite{Yang}, Yang has modified the argument of Inou-Shishikura to produce a class of maps, all with critical points of local degree three, which is invariant under parabolic and near-parabolic renormalization, allowing for generalization of the applications to unicritical cubic polynomials.
	Ch\'eritat has introduced smaller classes of maps which are invariant under parabolic and near-parabolic renormalization \cite{Cheritat22}; Ch\'eritat's construction is flexible enough so that the classes can have critical points of arbitrary local degree. This allows for further generalization of the applications of the Inou-Shishikura class. 
	
	All of the applications of the above invariant classes are limited to studying perturbations of maps with simple non-degenerate parabolic fixed points. However, it is well-known that similar parabolic implosion phenomena occur in the general setting.
	For a holomorphic map $f(z)$ which has a parabolic fixed point $z_0$ the local dynamics of $f$ near $z_0$ remain relatively tame: if the multiplier of the fixed point is a $q$-th root of unity then orbits are attracted towards $z_0$ along $\nu q$ distinct directions and repelled away from $0$ along $\nu q$ other distinct directions for some integer $\nu\geq 1$. The parabolic fixed point is called \textit{degenerate} or \textit{non-degenerate} if $\nu >1$ or $\nu= 1$ respectively. The parabolic implosion in the degenerate case is studied in \cite{Oudkerk}; there the bifurcation phenomenon is quite complicated and invariant classes under the corresponding near-parabolic renormalization have not been constructed.
	But in the general non-degenerate case the parabolic implosion is analogous to the simple case; Shishikura outlines the necessary modifications to generalize the analysis in \cite{shishikura_boundary}.

	There are three main goals of this paper. The first is to
	provide precise statements for the parabolic implosion phenomena in the non-degenerate setting; while Shishikura outlines the necessary generalizations in \cite{shishikura_boundary} the reader must reconstruct the statements themselves. Additionally, we correct a small error in \cite{shishikura_boundary} which is due to the fact that in the general setting, for a non-degenerate parabolic fixed point with multiplier $e^{2\pi i p/q}$,  Shishikura considers perturbations of with multiplier $e^{2\pi i (p+\alpha)/q}$. We instead consider a parameterization of the multiplier in terms of continued fractions; for more details on the difference see Remark \ref{rem:error}.
%
%

	The second goal of this paper is to construct invariant classes for the near-parabolic renormalizations in the general non-degenerate setting. 
	While the constructions employed in  \cite{shishikura} and \cite{Yang} are delicate and can not be immediately adapted, we observe that the construction in \cite{Cheritat22} is flexible enough to generalize. Thus for any root of unity we can construct  a class of maps invariant under the corresponding near-parabolic renormalization, allowing for further possible future extensions of the applications of the Inou-Shishikura class.

	The third goal of this paper is to compare the different  near-parabolic renormalizations that can be defined for a given map. 
	As an example, we consider a map 
	$h_0$ with a fixed point of multiplier $\lambda = e^{\frac{2\pi i \alpha}{n-\alpha}}$ for some small $\alpha$ and large integer $n\geq 0$.
	As $\lambda\approx e^{2\pi i /n}$, (in some cases) we can define a near-parabolic renormalization $\RR_{1/n} h_0$.
	But $\lambda\approx e^{2\pi i /n}$, we can also define a different near-parabolic renormalization $\RR_{1/n}h_0$. The map $\RR_{1/n}h_0$ has a fixed point with multiplier $e^{2\pi i \alpha}\approx 1$, so we can defined a secondary near-parabolic renormalization $\RR_1\RR_{1/n}h_0$. Our result, given more precisely in Theorem \ref{thm:comparison}, relates these renormalizations by: $$\RR_1 h_0 = \RR_1\RR_{1/n}h_0.$$

	This paper is organized as follows.
	In  $\S1$ and $\S2$ we present the parabolic and near-parabolic renormalizations respectively in the general non-degenerate setting. In  $\S3$ we define classes invariant under these renormalization operators. 
	In $\S4$ we relate different near-parabolic renormalizations.
	Let us remark that, outside of $\S4$, most of the ideas in this article are not new; the main novelty is collecting them together into one presentation. As such, we will usually avoid completely recreating classical arguments, instead referring to \cite{shishikura_1}, \cite{shishikura}, and \cite{Cheritat22} for the arguments in the simple non-degenerate case and highlighting the adjustments (or lack thereof) in the general setting. We will however, include a more detailed study of the construction of Fatou coordinates in the appendix.

	\noindent \textbf{Acknowledgments.} I am grateful to the referee for their many thoughtful comments and suggestions. This research was supported in part by the NSF and the Ford Foundation.

	\section{Parabolic  renormalization}\label{sec:parabolic}
	
	In this section we recall the theory of parabolic renormalization as introduced in \cite{shishikura_1}, but in the setting of general rational rotation numbers. 
	Some of the proofs are omitted;  we delay their discussion to the appendix.

	For any  $p/q\in \Q$, an analytic function $f$ defined on a neighborhood of zero satisfying $f(0) = 0$ and $f'(0) = e^{2\pi i p/q}$
	is said to have a $p/q$-parabolic fixed point at zero. If additionally
	$$f^q(z) = z+az^{q+1}+O(z^{q+2})$$
	near $z=0$ for some $a\in \C^*$, then the parabolic fixed point is called \textit{non-degenerate}.

\subsection{Petals and Fatou coordinates}

For any $t\in \R$, we will call a set of the form $$\{x+iy: x< ty+m\}\text{ or }\{x+iy: x>ty+m\}$$ with $m\in \R$ a \textit{left} or \textit{right half-plane} with \textit{tilt $t$} respectively.
For any $\lambda\in \C$, we denote $T_\lambda(z) = z+\lambda$.

For any $p/q\in \Q$, $t\in \R$, and analytic function $f$ defined in a neighborhood of zero, a \textit{parabolic $p/q$-flower near $0$ for $f$ with tilt $t$} is a collection $(P_j, \phi_j)_{j\in \Z/2q\Z}$ such that (see Figure \ref{parabolic petals fig}):

	\begin{figure}
	\begin{center}
		\def\svgwidth{6in}
		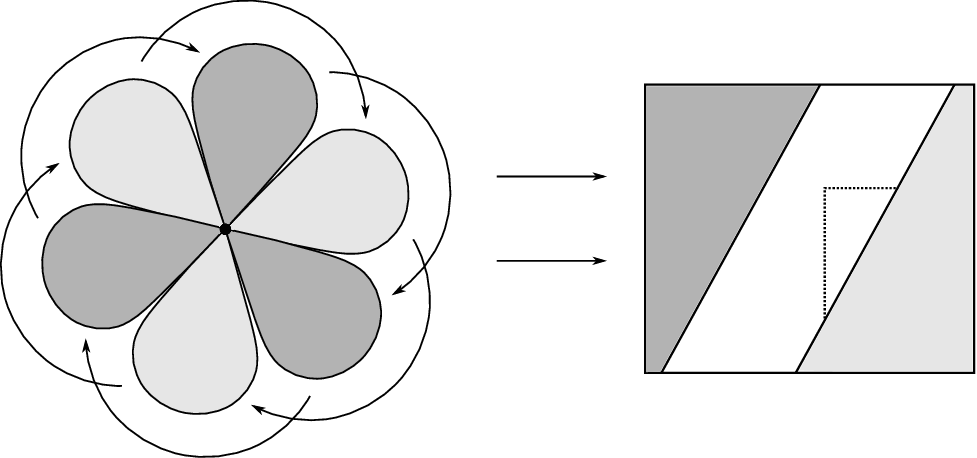
		\caption{A parabolic flower for $f$ with tilt $t$ and $p/q=-1/3$.}
		\label{parabolic petals fig}
	\end{center}
\end{figure}
\begin{enumerate}
	\item Each $P_j$ is a Jordan domain compactly contained in $\Dom(f^q)$, $\overline{P_j}\cap \overline{P_{j'}}= \{0\}$ for any $j\neq j'$, and the counter-clockwise circular ordering of the domains $P_j$ around zero is given by the ordering of $\Z/2q\Z$.
	\item Each $\phi_j$ is a conformal map defined on $P_j$  satisfying $$\phi_j\circ f^q = T_1\circ \phi_j$$
	wherever both sides of the equation are defined.
	\item If $j$ is even, then $\phi_j(P_j)$ is a right half-plane with tilt $t$ and
	$f^{nq+1}(P_{j})\subset P_{j+2p}$ for all large $n\geq 0$. 
	\item If $j$ is odd, then $\phi_j(P_j)$ is a left half-plane with tilt $t$ and  
	$f^{-(nq+1)}(P_{j+2p})\subset P_j$ for all large $n\geq 0$ and the continuous branch of $f^{-(nq+1)}$ fixing zero.
	\item Any forwards or backwards orbit under $f^q$ that converges towards $0$ is eventually contained in some $P_j$. 
\end{enumerate}

For a parabolic flower $(P_j, \phi_j)$, the domains $P_j$ are called the \textit{petals} of the flower and the maps $\phi_j$ are called the \textit{Fatou coordinates}; if $j$ is even or odd the petals and Fatou coordinates are called \textit{attracting} or \textit{repelling} respectively.
The Fatou coordinates are uniquely defined up to post-composition by a translation, see for example \cite[Theorem 10.9]{Milnor}, and a particular choice of Fatou coordinates is called a \textit{normalization}.

\begin{remark}
	Throughout the literature there are several different definitions of parabolic flowers and petals, see for example \cite[Chapter 10]{Milnor}. We choose the definition given above because it is convenient to explicitly control the geometry of petals.
\end{remark}

The following classical result of Leau \cite{leau} and Fatou \cite{fatou_flower} asserts that non-degenerate parabolic fixed points produce parabolic flowers:

\begin{theorem}\label{flower}
	If $f$ is a holomorphic map defined in a neighborhood $V$ of zero and with a non-degenerate $p/q$-parabolic fixed point at zero, then for any $t\in \R$ there is a $p/q$-parabolic flower for $f$ with tilt $t$ inside $V$.
\end{theorem}

Let us now fix a map $f$ with a non-degenerate $p/q$-parabolic fixed point at zero and a parabolic $p/q$-flower $(P_j, \phi_j)$.
While the parabolic flower is not uniquely defined, any two flowers are closely related:
\begin{prop}\label{prop:petals}
	If $(\tilde{P}_j, \tilde{\phi_j})$ is another parabolic flower for $f$, then we can normalize the Fatou coordinates so that  $\phi_j^{-1}= \tilde \phi_{j+k}^{-1} $ on $\phi_j(P_j)\cap \tilde \phi_{j+k}(\tilde P_{j+k}) $ for all $j$ and some even $k$.
\end{prop}
Let us also note that the petals cannot spiral around zero:
\begin{prop}\label{no spiral}
	There exists some $M>0$ such that there is a continuous branch of $\log$ defined on each $P_j$ such that $\log(P_j)/2\pi i$ is contained in a strip with tilt zero and width $M$. 
\end{prop}

	The attracting Fatou coordinate $\phi_0$ can be analytically extended to a map 
	$$\rho(z) := \phi_{0}\circ f^{n}(z)-\lfloor n/q\rfloor$$
	for any integer $n\geq 0$ and $z\in \Dom(f^{n})$ such that $f^{n}(z)\in P_{0}$, where $\lfloor x \rfloor$ is the largest integer less than or equal to $ x$. 
	The map $\rho$ is called an \textit{extended attracting Fatou coordinate} and semi-conjugates $f^q$ to $T_1$.	
	Points in the domain of $\rho$ can be labeled by how many iterates modulo $q$ it takes for them to enter $P_0$: for any $0\leq k < q$ we will call the set of all $z\in \Dom(\rho)$ such that $f^{nq+k}(z)\in P_0$ the \textit{$k$-th section} of $\Dom(\rho)$.  
	
	The inverse of the repelling Fatou coordinates $\phi_{\pm 1}$ can be similarly extended to
	$$\chi_\pm(w) := f^{nq}\circ \phi_{\pm 1}^{-1}(w-n)$$
	for any $n\geq 0$ and $w$ with  $\phi_1^{-1}(w-n)\in P_{\pm 1}\cap \Dom(f^{nq})$.
	The maps $\chi_\pm$ are called \textit{extended repelling Fatou parameters} and semi-conjugate $T_1$ to $f^q$.

	\begin{remark}
		The extended maps $\rho$ and $\chi_\pm$ depend on the normalization of the Fatou coordinates; changing the normalization corresponds to post-composing and pre-composing by a translation respectively. 
		Note that Proposition \ref{prop:petals} implies that these extended maps only depend on the normalization and not the choice of flower; we will see this fact more explicitly in Proposition \ref{prop:unqiue extensions} below.   
	\end{remark}
	
	\begin{remark}
		While we have only defined extensions for the Fatou coordinates $\phi_0$ and $\phi_{\pm 1}^{-1}$, we could similarly extend the other Fatou coordinates.
	\end{remark}
	
	\begin{remark}
		In some cases it is more convenient to define the extension of the attracting Fatou coordinate so that $$\rho(z) = \phi_0\circ f^n(z)-n/q.$$ 
		With such a definition, $\rho$ semi-conjugates $f$ to the translation $z\mapsto 1/q$. 
		 Our choice of definition is made so that 
		Proposition \ref{similar horn maps} below holds.
	\end{remark}

	\subsection{Horn maps}
	
	For the flower $(P_j, \phi_j)$ and corresponding extensions $\rho$ and $\chi_\pm$ as above,  the functions
	$$H_{\pm} : = \rho\circ \chi_{\pm}$$
	are called \textit{horn maps} for $f$.
	\begin{proposition}\label{horn map parabolic}
		The horn maps $H_{\pm}$ commute with $T_1$ and are analytic on their domains, which contain $\{w\in \mathbb{C}: |\emph{Im}\,w|> \eta_0\}$ for some $\eta_0>0$. 
		For any normalization of $\phi_0$, there exist unique normalizations of $\phi_{\pm 1}$  such that 
		$H_{\pm}(w) -w$ tends to zero when  $\emph{Im}\, w\to \pm \infty$ and to a constant when $\emph{Im}\,w\to \mp\infty$.
	\end{proposition}

	With Proposition \ref{horn map parabolic}, we can now specify some normalizations of the Fatou coordinates.
	Fixing some $z_0$ in the $0$-th section of $\Dom(\rho)$, we can normalize $\phi_0$ so that $\rho(z_0) = w_0$. 
	We can then take the normalizations of $\phi_{\pm 1}$ as in Proposition \ref{horn map parabolic}, so $H_\pm(w)-w\to 0$ when $\text{Im}\,w\to \pm \infty$.
	With this choice, we will say that the Fatou coordinates are \textit{normalized by $(z_0, w_0)$}. 
	\begin{prop}\label{prop:unqiue extensions}
		The maps $H_\pm$ are uniquely determined by $f$ and the normalization.
	\end{prop}
	
	\begin{proof}
		Let $(P_j, \phi_j)_{j\in \Z/2q\Z}$ and $(\tilde{P}_j, \tilde \phi_j)_{j\in \Z/2q\Z}$ be two flowers for $f$ with corresponding 
		Fatou extensions $\rho, \chi_\pm$ and $\tilde \rho, \tilde \chi_\pm$ and 
		horn maps $H_\pm$ and $\tilde{H}_\pm$ respectively.
		
		It follows from Proposition \ref{prop:petals} that we can choose the Fatou coordinates $\tilde \phi_j$ so that 
		there is some even integer $s$ with $\phi_j^{-1}=\tilde \phi_{j+s}^{-1}$ on $\phi_j(P_j)\cap \tilde\phi_{j+s} (\tilde P_{j+s}).$
		In order to choose the same normalizations for $H_\pm$ and $\tilde{H}_\pm$, there must be some $z_0$ in the $0$-th sections of both $\rho$ and $\tilde \rho$, so $s = 0$.
		It then follows immediately from the definitions of the extensions that $\rho = \tilde \rho$ and $\chi_\pm = \tilde\chi_\pm$, hence $H_\pm= \tilde H_\pm$. 
	\end{proof}

	Fixing some normalization for $H_\pm$ as above, 
	the two horn maps $H_+$ and $H_-$ also differ by translations:

	\begin{prop}\label{similar horn maps}
		There is some $\lambda\in \C$ such that $$H_{+}(w) -H_{-}\circ T_\lambda(w) \in \{0, 1\}$$
		for all $w\in \Dom (H_{+}).$ 
	\end{prop}
	
	\begin{proof}
		Theorem \ref{flower} implies
		that there are some integers $0\leq j < q$ and $n\geq 0$ such that $f^{-(nq+j)}(P_{-1})\subset P_{+1},$ using the inverse branch that fixes zero.
		The map $\phi_{-1}\circ f^{-(nq+j)}$ is also a Fatou coordinate for $f^q$ on $P_{+1}$, so it follows from the uniqueness of Fatou coordinates that there is some  $\lambda\in \C$ such that $$\phi_{-1}\circ f^{-(nq+j)} =  T_{\lambda-n}\circ \phi_{+1}$$ on $P_{+1},$ so 
		$$f^{j}\circ \chi_- = \chi_+\circ T_{-\lambda}.$$
		It follows from our definition of $\rho$ that $\rho\circ f^j(z)-\rho(z)\in \{0, 1\}$ for any $z$ in the domain, the proposition then immediately follows from the definitions of $H_{\pm}.$
	\end{proof}
	
	Throughout this article we will consider analytic maps in the compact-open topology with varying domains, i.e. a neighborhood of $f$ is a set of the form 
	$$\left\lbrace h:\Dom(h)\to \mathbb{C} \;\middle|\;
	\begin{tabular}{@{}l@{}}
		$h \text{ is analytic on }\Dom(h)\supset K, \text{ and }$\\
		$|f(z) - h(z)|< \epsilon \text{ for all }z\in K$
	\end{tabular}
	\right\rbrace$$
	for some $\epsilon>0$ and compact set $K\subset \Dom(f).$ 
	In this topology, with a fixed choice of normalization, the horn maps depend continuously and analytically on $f$:
	\begin{prop}\label{continuous dependence}
		Fixing a normalization by some $(z_0, w_0)$,  		
		$\rho, \chi_\pm$ and $H_\pm$ all depend continuously and analytically on $f$.
	\end{prop}

	\subsection{Parabolic renormalization}
	Denoting  $\Exp_\pm(w) = e^{\pm 2\pi i w}$, we define a \textit{parabolic renormalization} of $f$  to be a map of the form
	\begin{align*}
		\mathcal{R}_{\delta}^\pm f&:= \Exp_\pm\circ  H_{\pm }\circ T_{\delta}\circ(\Exp_\pm)^{-1}
	\end{align*}
	for some $\delta\in \mathbb{C}$.
	Proposition \ref{horn map parabolic} implies that 
	$\mathcal{R}_{\delta}^\pm f$ is defined on punctured neighborhoods of zero and infinity in $\hat{\mathbb{C}}$ and can be continuously extended by setting
	$\mathcal{R}_\delta^\pm f(0) = 0$ and $\mathcal{R}_\delta^\pm f(\infty) = \infty.$ 
	As $H_\pm(w) -w\to 0$ when $\text{Im}\,w\to \pm \infty$,
	we can compute the derivatives $$(\mathcal{R}_\delta^\pm f)'(0) = \Exp_\pm(\delta).$$
	Note that Proposition \ref{similar horn maps} allows to relate parabolic renormalizations by 
	$$\RR_{\delta}^+f(w) = \frac{1}{\RR_{\delta+\lambda}^-f(1/w)}$$
	for some $\lambda\in \C$ depending only on $f$.
	 To simplify our notation, we will usually write $\Exp = \Exp_+$.
	
	Proposition \ref{prop:unqiue extensions} implies that the renormalization $\RR_\delta^\pm f$ 
	depends only on the normalization of Fatou coordinates and not on the choice of parabolic flower. 
	In the Section \ref{sec:classes}  we will restrict to a family of maps with  canonical choices of normalization, so the renormalizations $\RR_\delta^\pm f$ will be uniquely defined.	
	It  follows from Proposition \ref{continuous dependence} that, for a fixed normalization, the renormalizations $\RR_\delta^\pm f$ depend continuously and analytically on $f$.
	
	\begin{remark}
		The definition of parabolic renormalization we give here differs from the definition in  \cite{shishikura}, where \textit{the} parabolic renormalizations is defined as $\mathcal{R}_0^\pm f$. While that definition is well-suited for studying maps with non-degenerate $0/1$-parabolic fixed points, for our more general setting it is important to consider more possible values for the multiplier. Note however that the maps $\RR_\delta^\pm f$ need not have parabolic fixed points.
	\end{remark}
	
	\begin{remark}
		Choosing two different normalizations for the horn maps corresponds to conjugating $H_\pm$ by a translation, and hence conjugating the renormalizations $\RR_\delta^\pm f$ by a linear map. If we wanted to work without fixing a normalization for the horn maps, the parabolic renormalizations would be uniquely defined up to linear conjugacy.
	\end{remark}

	\begin{remark}
		We could have  defined $\mathcal{R}^\pm_\delta f$ so that it is semi-conjugate 
		to $T_\delta\circ H_\pm$ instead of $H_\pm\circ T_\delta$. This distinction is purely aesthetic; in this article it is convenient to have the critical values of $\mathcal{R}_\delta^\pm f$ not depend on $\delta$, in some other cases it may instead convenient to have the domain of $\mathcal{R}_\delta^\pm f$ not depend on $\delta.$
	\end{remark}

	\subsection{Lavaurs maps}
	
	While we will focus primarily on the parabolic renormalization, there is an alternative formulation of parabolic implosion in terms of \textit{Lavaurs maps}. 
	A Lavaurs map for $f$ is a map of the form 
	$$L_\delta^\pm= \chi_\pm\circ T_\delta\circ \rho$$
	for $\delta\in \C$.
	These Lavaurs maps are analytic and satisfy 
	$$L_{\delta}^\pm\circ f^q= L_{\delta+1}^\pm= f^q\circ L_{\delta}^\pm$$
	wherever both sides of the equation are defined. Note that a Lavaurs map
	is defined on all of $\Dom(\rho)$ when $f(\Dom(f))\subset \Dom(f)$, otherwise the domain may be smaller. However, for any $z\in \Dom(\rho)$ and $\delta\in \C$, $L_{\delta-n}^\pm(z)$ is defined for all sufficiently large $n\geq 0$.
	
	A Lavaurs map $L_{\delta}^\pm$ is semi-conjugate to $H_{\pm}\circ T_\delta$ by both $\rho$ and $\chi_\pm\circ T_\delta$, and thus is also semi-conjugate to $\RR_\delta^\pm f$.
	In different applications, it is may be convenient to focus on either  Lavaurs maps,  horn maps, or  parabolic renormalizations; the commutative diagram below shows explicitly how they are all related:
	
	\begin{center}	
	\begin{tikzcd}
		\Dom(L_\delta^2) \arrow[d, "\rho"'] \arrow[r, "L_\delta^\pm"]                                                        & \Dom(L_\delta) \arrow[r, "L_\delta^\pm"] \arrow[d, "\rho"']    & \C \\
		\Dom(H_\pm\circ T_\delta) \arrow[d, "\Exp_\pm"] \arrow[ru, "\chi_\pm\circ T_\delta"] \arrow[r, "H_\pm\circ T_\delta"] & \C \arrow[d, "\Exp_\pm"] \arrow[ru, "\chi_\pm\circ T_\delta"'] &    \\
		\Dom(\RR_\delta^\pm f) \arrow[r, "\RR_\delta^\pm f"]                                                                 & \C^*                                                           &   
	\end{tikzcd}
	\end{center}

	\section{Near-parabolic renormalization}\label{sec:near-parabolic}
	
	Let us now recall the theory of near-parabolic renormalization as introduced in \cite{shishikura_1}.
	As in the previous section, we will  delay some proofs to the  appendix.

	For $p/q\in \Q$ and $g$ an analytic map with $g(0) =0$ and $g'(0)$ close to $e^{2\pi i p/q}$, in this section we will see that some dynamics similar to the parabolic case can persist for $g$. 
	The precise dynamics will depend on  some arithmetic properties of $p/q$, specifically the \textit{modified continued fraction} expansion, which is used in  \cite{shishikura}
	to study near-parabolic renormalization in the $p/q=0/1$ case. 
	
	\subsection{Modified continued fractions}	
	For any rational number $p/q$, a \textit{modified continued fraction expansion} of $p/q$ is a sequence of pairs $\kappa = ( a_n, \varepsilon_n)_{n=0}^N$ such that 
	$a_0\in \Z$,  $a_n\in \Z_{\geq 2}$ for all $n\geq 1$, $\epsilon_n=\pm 1$ for all $n$, and 
	\begin{equation*}
		\epsilon_0p/q=  a_0+\cfrac{\varepsilon_1}{a_1+\cfrac{\varepsilon_2}{a_2+\cfrac{\varepsilon_3}{\ddots+\cfrac{\varepsilon_N}{a_N}}}}.
	\end{equation*}
	We will call $N\geq 0$ the \textit{length} of the modified continued fraction expansion; we denote by $\Q_N$ the set of all rational numbers that have a modified continued fraction expansion of length at most $N$.
	We define the \textit{signature} of the modified continued fraction to be 
	$$\mathfrak{S}(\kappa) = (-1)^N\prod_{n=0}^N\varepsilon_n.$$

	We define the M\"obius transformation $\mu_\kappa: \hat{\mathbb{C}}\to \hat{\mathbb{C}}$ by 
	\begin{equation*}\label{definition of mu}
		\epsilon_0\mu_{\kappa}(z) :=a_0+ \cfrac{\varepsilon_1}{a_1+\cfrac{\varepsilon_2}{a_2+\cfrac{\varepsilon_3}{\ddots+\cfrac{\varepsilon_N}{a_N+z}}}}.
	\end{equation*}
	Let us denote by $0\leq q'_\kappa < q$ the unique integer satisfying $pq'_\kappa + \mathfrak{S}(\kappa)\equiv 0 \mod q$.
	We have the following alternative expression of $\mu_\kappa$:
	\begin{prop}\label{prop:mobius form}
		We have
		$$\mu_{\kappa}(z) = \frac{p}{q}+\frac{\mathfrak{S}(\kappa)z}{q(q+ q_\kappa' z)}.$$
	\end{prop}
	
	\begin{proof}
		For all $0\leq n \leq N$, set  $\kappa_n = ( a_m, \varepsilon_m)_{m=0}^n$. 
		For all $0\leq n \leq N$, let $p_n/q_n$ be the rational number with modified continued fraction expansion $\kappa_n$; for $p_{-1} = \epsilon_0$ and $q_{-1} = 0$ these numbers satisfy the relation
		$$p_n= a_n p_{n-1} +\varepsilon_n p_{n-2}\text{ and }q_n= a_n q_{n-1} +\varepsilon_n q_{n-2}.$$
		As $$q_np_{n -1}- p_nq_{n-1} = -\varepsilon_n(q_{n-1}p_{n-2}- p_{n-1}q_{n-2}),$$
		we can show by induction that 
		\begin{equation}\label{eq:signature}
			q_np_{n-1}-p_nq_{n-1} = \mathfrak{S}(\kappa_n).
		\end{equation}		
		As $\mu_{\kappa_n}(z) = \mu_{\kappa_{n-1}}\left(\frac{\varepsilon_n}{a_n+ z}\right)$, we can also show by induction that 
		$$\mu_{\kappa_n}(z) = \frac{p_n+p_{n-1}z}{q_n+q_{n-1}z} =  \frac{p_n}{q_n} + \frac{\mathfrak{S}(\kappa_n)z}{q_n(q_n+q_{n-1}z)} .$$
		As \eqref{eq:signature} implies $q_{N-1} = q_{\kappa}'$, the proof is complete.
	\end{proof}
	
	Proposition \ref{prop:mobius form}  implies that $\mu_\kappa$ is uniquely determined by the signature of $\kappa$. We denote $\mu_{p/q}^\pm= \mu_\kappa$ for any modified continued fraction expansion $\kappa$ for $p/q$ with $\mathfrak{S}(\kappa) = \pm 1$. We will also denote $q_\pm' = q_\kappa$ in this case; while $q_\pm'$ depends on $p/q$ and not just $q$ we will not include this dependence in our notation when the choice of $p/q$ is clear.
	Using Proposition \ref{prop:mobius form}, we can directly compare $\mu_{p/q}^+$ and $\mu_{p/q}^-$. 
	
	\begin{prop}
		We have:
		$$\mu_{p/q}^- (z) = 
		\begin{cases}
			\mu_{p/q}^-(-z) & \text{ if } q= 1;\\
			\mu_{p/q}^+\left(\frac{-z}{1+z}\right) &\text{ if }q>1.
		\end{cases}
		$$
	\end{prop}
	
	\begin{proof}
		If $q = 1$, then it follows from the definition that $q_\pm' = 0$, so $\mu_{p/q}^-= \mu_{p/q}^+$ by Proposition \ref{prop:mobius form}.
		If $q> 1$, then it follows from the definition that $q_{-}' = q-q_+'$, so 
		$$\mu_{p/q}^-(z) = \frac{p}{q}+ \frac{-z}{q(q(1+z)-q_+'z)}= \mu_{p/q}^+\left(\frac{-z}{1+z}\right)$$
		as desired.
	\end{proof}
	
	Note that replacing $\epsilon_0$ with $-\epsilon_0$ in $\kappa$ corresponds to replacing $\mathfrak{S}(\kappa)$ and $p/q$ with $-\mathfrak{S}(\kappa)$ and $-p/q$ respectively. Hence $-\mu_{p/q}^\pm = \mu_{-p/q}^\mp$.
	
	\begin{remark}
		In the literature a modified continued fraction  is usually assumed to have $\epsilon_0= +1$, in which case any non-integer rational number has exactly modified continued fraction expansions. However, the two expansions may have drastically different lengths: for example we can write 
		$$0+\frac{1}{5}= 1+ \cfrac{-1}{2+\cfrac{-1}{2+\cfrac{-1}{2+\cfrac{-1}{2}}}}.$$
		By allowing $\epsilon_0=-1$, any $p/q\in \Q_N$ has a modified continued fraction expansions of length $N$ for either signature.
	\end{remark}

	\subsection{Petals and Fatou coordinates}
	

For any $t\in \R$, 
we will call a set of the form 
$$\{x+iy:ty+a< x< ty+b\}$$
with $a < b$ a \textit{strip with tilt $t$ and width $b-a$}.
For any $p/q\in \Q$, $t\in \R$, and analytic function $g$ defined in a neighborhood of zero, a \textit{near-parabolic $(p/q, \pm)$-flower near $0$ for $g$ with tilt $t$} is a collection $(P_j, \phi_j)_{j\in \Z/2q\Z}$ such that (see figure \ref{parabolic perturned fig}):

\begin{enumerate}
	\item 
	For any even $j$, each $P_j = P_{j\pm 1}$ is a Jordan domain compactly contained in $\Dom(g^q)$. For any even $j$ and $j'$,   $\overline{P_j}\cap \overline{P_j'}= \{0\}$. The counter-clockwise circular ordering of the domains $P_j$  around zero with $j$ even is given by the ordering of $2\Z/2q\Z$.
	\item Each $\phi_j$ is a conformal map defined on $P_j$ satisfying
	$$\phi_j\circ g^q= T_1\circ \phi_j$$
	wherever both sides of the equation are defined. 
	\item Each $\phi_j(P_j)$ is a strip with tilt $t$ and width at least $2$. 
	If $j\neq 0$, then $g(P_j) = P_{j+2p}.$
	\item Each $P_j$ has a non-zero fixed point of $g^q$ on its boundary, and $z\in P_j$ tends to zero when $\text{Im}\,\phi_j(z)$ tends to $\pm\infty$.
\end{enumerate}

\begin{figure}
	\begin{center}
		\def\svgwidth{6in}
		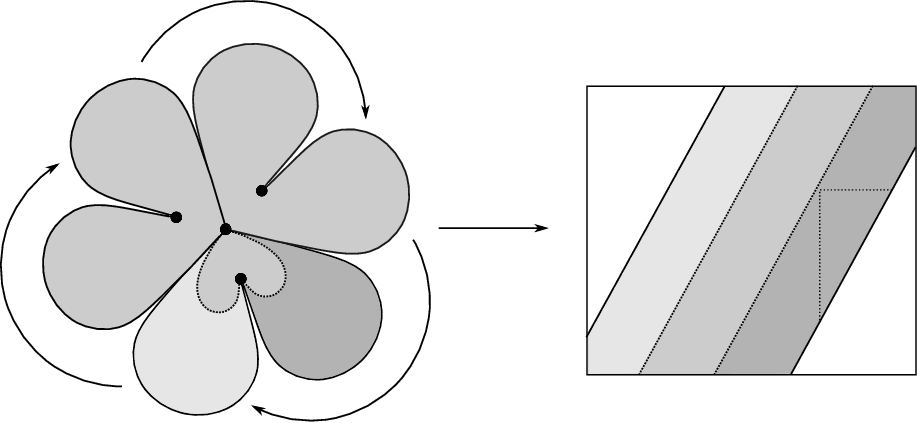
		\caption{A $(-1/3, +)$-near-parabolic flower for $g$ with tilt $t$. The spiraling of the petals around zero is controlled by Proposition \ref{prop:petals perturbed}. The points entering and exiting $P_0=P_1$ are shown in lighter and darker gray respectively.}
		\label{parabolic perturned fig}
	\end{center}
\end{figure}

As in the parabolic setting, the domains $P_j$ are called petals and the maps $\phi_j$ are called Fatou coordinates; the Fatou coordinates are again unique up to post-composition by translation.

For any $r>0$, we denote $$A_r= \{x+it:  x>|4ry|\}.$$
The following theorem, for which the $p/q=0/1$ case is proved in \cite{shishikura_1} and the general case in \cite{shishikura_boundary}, shows that perturbations
of parabolic maps have near-parabolic flowers:

\begin{theorem}\label{thm:near-parabolic flowers}
	Let $f$ be a holomorphic map with a non-degenerate $p/q$-parabolic fixed point at zero and  let $V$ be a neighborhood of $0$.
	For any $r>0$,
 	if $g$ is an analytic map with 
	$$g(0) = 0 \text{ and }g'(0) = \Exp\circ \mu_{p/q}^\pm(\alpha)$$
	for some $\alpha\in A_r$, then $g$ has a near-parabolic $(p/q, \pm)$-flower $(P_j, \phi_j)_{j\in \Z/2q\Z}$ with tilt $t$ such that 
	each $\phi_j$ depends continuously and holomorphically on $g$ for any $|t|< r$. Moreover, there is a parabolic flower $(P^f_j, \phi^f_j)_{j\in \Z/2q\Z}$ in $V$ for $f$   such that $\phi_j\to \phi_j^f$ when $g\to f$.
\end{theorem}

For the rest of this section let us fix some $p/q\in \Q$ and map $f$ with a  non-degenerate $p/q$-parabolic fixed point at zero.
Fixing some $t\in \R$ and $s>|t|$, let 
$g$ be an analytic map close to $f$ with $g(0) = 0$ and $g'(0) = \Exp\circ \mu_{p/q}^\pm(\alpha)$ for some $\alpha\in A_s$ and let 
$(P_j, \phi_j)$ be a near-parabolic $(p/q, \pm)$-flower for $g$ with tilt $t$ as in Theorem 
\ref{thm:near-parabolic flowers}. So each $\phi_j$ depends continuously and holomorphically on $g$ and tends to $\phi_j^f$ when $g\to f$ for some parabolic $p/q$-flower $(P_j^f, \phi_j^f)$ for $f$.
We will call $(P_j, \phi_j)$ a near-parabolic flower for $g$ \textit{relative to} $f$.
Throughout this section we will use the symbols $\pm$ and $\mp$ in reference to the corresponding choice for $g'(0)$.

Similarly to parabolic flowers,  near-parabolic flowers are not uniquely defined but  different flowers are  related:

\begin{prop}\label{prop:petals perturbed}
	If $(\tilde{P}_j, \tilde \phi_j)$ is another near-parabolic flower for $g$ relative to $f$, then we can normalize so that there is an even integer $k$ such that $\phi_j^{-1}= \tilde{\phi}_{j+k}^{-1}$ on the non-empty set $\phi_j(P_j)\cap \tilde\phi_{j+k}(\tilde P_{j+k})$ for all $j$. 
\end{prop}
Additionally, we observe that the near-parabolic flowers only weakly depend on $f$:
\begin{prop}\label{prop:near-parabolic independence}
	Let $\tilde{f}$ be another analytic map with a non-degenerate $p/q$-parabolic fixed point at zero. If $f$ is sufficiently close to $\tilde{f}$, then $(P_j, \phi_j)$ is also near-parabolic flower for $g$ relative to $\tilde{f}$.
\end{prop}
\begin{remark}
	In light of Proposition \ref{prop:near-parabolic independence}, we might wonder if the near-parabolic flower $(P_j, \phi_j)$ can be defined without assuming that $g$ is close to some particular $f$. In some cases, we can remove the dependence on $f$ entirely;
	 this is one of the consequences of Theorem \ref{near-parabolic invariance} in the next section.
\end{remark}

Unlike parabolic flowers, near-parabolic flowers may spiral around zero. This spiraling is controlled by the tilt and argument of $\alpha$:

\begin{prop}\label{prop:tilt}
	There exists some $M>0$ depending only on $f, s$, and $t$ such that if $g$ is sufficiently close to $f$ then there is a branch of $\log$ defined on each $P_{j}$ such that $\log (P_j)/2\pi i$ is contained in a strip with tilt $t'$ and width $t'|\log |\alpha||+M$, where  
	$$t' =\frac{t\emph{Re}\, \alpha'\mp\emph{Im}\, \alpha'}{ t\emph{Im}\,\alpha'\pm\emph{Re}\, \alpha'}\text{ and }\alpha' = \frac{q\alpha}{q+q_\pm'\alpha}.$$
\end{prop}

%
	
	\begin{remark}
		The spiraling of the petals around the non-zero fixed points of $g^q$ can be similarly computed: the petals lift to strips with tilt $-t'$ instead.
	\end{remark}

	Let $\rho^f$ and $\chi_\pm^f$ be the extensions of $\phi_0^f$ and $(\phi_{\pm 1}^f)^{-1}$ respectively as in the previous section; we can similarly extend $\phi_0$ and $\phi_{\pm 1}^{-1}$.
	Let $W$ denote the width of the strip $\phi_0(P_0)$.
	We will say that a point $z$ is \textit{entering} or \textit{exiting} $P_0=P_{\pm 1}$ if 
	$$\phi_0(z)-W/3\notin \phi_0(P_0) \text{ or }\phi_{\pm 1}(z)+ W/3\notin \phi_{\pm 1}(P_{\pm 1})$$
	respectively (see Figure \ref{parabolic perturned fig}). 	Note that by definition no point can be both entering and exiting $P_0$ when $g$ is close to $f$.
	We extend $\phi_0$ by defining
	$$\rho(z) = \phi_0\circ g^{n}(z)-\lfloor n/q\rfloor$$
	for any point $z$ and integers  $0 \leq n < qW/3$ such that $g^{n}(z)$ is entering $P_0$.
	For any $0\leq k < q$, we define the $k$-th section of $\Dom(\rho)$ to be the set of all $z$ so that $n \equiv k \mod q$.	
	We extend $\phi_{\pm 1}^{-1}$ by 
	$$\chi_\pm(w):= g^{nq}\circ \phi_{\pm 1}^{-1}\circ T_{-n}(w).$$
	for any $w$ such and integer $0\leq n< W/3$ such that $\phi_{\pm 1}^{-1}(w-n)\in \Dom(g^{nq})$ is exiting $P_{\pm 1}$.

	\begin{proposition}\label{prop:extensions}
		The maps $\rho$ and $\chi_\pm$ are well-defined, analytic, and converge to $\rho^f$ and $\chi^f$ respectively when $g\to f$. 
	\end{proposition}
	
	\begin{proof}
		It follows from the fact that the petals $P_j$ are all disjoint, and our use of the width $W/3$ in the definition of entering, that if $g^{n_1}(z)$ and  $g^{n_2}(z)$ are both  entering $P_0$ for some $n_1\leq n_2< qW/3$, then $n_2-n_1\in q\Z$. This immediately implies that $\rho$ is well-defined. The proof that $\chi_\pm$ is well-defined is similar.
		The analyticity of $\rho$ and $\chi_\pm$ holds automatically, and the convergence follows from the convergence $\phi_0\to\phi_0^f$ and  $\phi_{\pm 1}\to \phi_{\pm 1}^f$.  
	\end{proof}

%
%

	\subsection{Horn maps}
	
	The function $$H_\pm:= \rho\circ \chi_\pm$$
	is called a  \textit{horn map for $g$ relative to $f$}.

		\begin{proposition}\label{horn maps perturbed}
			Using the functional equation 
			$$T_1\circ H_\pm= H_\pm\circ T_1,$$
			when $g$ is close to $f$ we can analytically extend  $H_\pm$ to a $T_1$-invariant domain that contains $\{w\in \mathbb{C}: |\emph{Im}\,w|> \eta_0\}$ for some $\eta_0$ depending only on $f$.
			For any normalization of $\phi_0$, there is a unique normalization of $\phi_{\pm 1}$ so that $H_\pm(w)-w$ tends to zero when $\emph{Im}\,w\to \pm \infty$ and to a constant when $\emph{Im}\,w\to \mp\infty$.
	\end{proposition}
	
	Let $H_\pm^f = \rho^f\circ \chi_\pm^f$ be the corresponding horn map for $f$.
	Recall that $H_\pm^f$ is normalized so that $\rho^f(z_0)=w_0$ for some $z_0$ in the $0$-th section of $\Dom(\rho^f)$ and $w_0\in \C$, and so that $H_\pm^f(w)-w\to 0$ when $\text{Im}\,w\to \pm \infty$. 
	When $g$ is close to $f$, Proposition \ref{horn maps perturbed} implies that we can similarly normalize the horn map $H_\pm$ so that $\rho(z_0)= w_0$ and $H_\pm(w)-w\to 0$ when $\text{Im}\,w\to \pm \infty$. 
	Fixing this normalization, we can explicitly compute the difference between $\phi_0$ and $\phi_{\pm 1}$, which is called the \textit{phase} of $g$:
	\begin{prop}\label{prop:phase}
		Normalizing the Fatou coordinates for $(P_j, \phi_j)$ by some $(z_0, w_0)$, we have
		$$\phi_{\pm 1}= T_{-1/\alpha}\circ \phi_0.$$
		Moreover, $H_\pm$ depends continuously and holomorphically on $g$ and converges to $H_\pm^f$ when $g\to f$.
	\end{prop} 
	
	\begin{rem}\label{rem:error}
		Proposition \ref{prop:phase} is exactly the point where our description differs from  \cite{shishikura_boundary}. In \cite{shishikura_boundary}, Shishikura considers perturbations where $g'(0) = \Exp((p+\beta)/ q)$ with $\beta\in A_{s}$ and claims that when normalized we have $\phi_{\pm 1}= T_{-1/\beta}\circ \phi_0$ (while not  \textit{explicit}, the claim is implicit in the the modifications to \cite[(4.2.4)]{shishikura_boundary} in \cite[\S 7]{shishikura_boundary}). However, using Proposition \ref{prop:mobius form} to solve for $\alpha$ in terms of $\beta$, Proposition  \ref{prop:phase} implies that in this case  we instead have
		$$\phi_{\pm 1}(z) = \phi_0(z)-\frac{1}{q\beta}+\frac{q_+'}{q}.$$
		This discrepancy  is in many cases minor, for example it  has no effect when $q= 1$ or in \cite{shishikura_boundary}. However when considering repeated near-parabolic renormalization this discrepancy is significant,  most noticeably in Theorem \ref{thm:comparison} and its corollaries.
	\end{rem}

	We saw in Proposition \ref{prop:unqiue extensions} that the horn maps for $f$ depend on the normalization of the Fatou coordinates but not on the choice of flower.  
	The horn map for $g$ relative to $h$ is similarly independent on the choice of flower, though the independence is slightly weaker:
	
	\begin{prop}
		Let $X\subset \Dom(H_\pm^f)$ be a compact set. Any two continuous and holomorphic choices of horn map $H_\pm$ for $g$ relative to $f$ agree on $X$ when $g$ is close to $f$.
	\end{prop}
	
	\begin{proof}
		The convergence $\phi_0\to \phi_0^f$ and $\phi_1\to \phi_1^f$ implies that any compact set in $P_0^f$ or $P_1^f$ is entering or exiting $P_0=P_1$ respectively when $g$ is  close to $f$. With this observation, the proof is the same as Proposition \ref{prop:unqiue extensions}, using Proposition \ref{prop:petals perturbed} instead of Proposition \ref{prop:petals}.
	\end{proof}
	
	While the dynamics of the horn map $H_\pm^f$ are not related to the dynamics of $f$, the horn map dynamics of $H_\pm$ almost semi-conjugate to high iterates of $g$:

	\begin{prop}\label{lifing the renormalized dynamics}
		Let $z, z'$ be two points in $P_0$ and set $w= \phi_0(w)$, $w'= \phi_0(z')$. If 
		$$H_\pm\circ T_{n- 1/\alpha}(w) = w'$$
		for some integer $n$, then there are integers $m\geq 0$ and $0\leq k < q$ such that either
		\begin{align*}
			g^{(n+m)q+k}(z)&= g^{mq}(z') &&\hspace{-3cm}\text{ if }n\geq 0, \text{ or}\\
			g^{mq+k}(z) & = g^{(m-n)q}(z') && \hspace{-3cm}\text{ if }n\leq 0.
		\end{align*}
		If $z'$ is exiting $P_0$, then $m = 0$. If $|\emph{Im}\,w|$ is sufficiently large, then $m = 0$ and 
		$$k = \begin{cases}
			q_\pm' & \text{ if }\pm\emph{Im}\, w>0,\\
			0 &\text{ if }\mp\emph{Im}\, w< 0.
		\end{cases}$$
	\end{prop}
	
	\begin{proof}
		First we note that as $z\in P_0$, it follows from the definition of $H_\pm$ that there is some integer $j\geq 0$ satisfying
		\begin{equation}\label{eq:horn map}
			H_\pm\circ T_{n-1/\alpha}(w) = T_{n-j}\circ \rho \circ \chi_\pm\circ T_{j- 1/\alpha}\circ \phi_{0}(w).
		\end{equation}
		Indeed, the definition of $H_\pm$ guarantees \eqref{eq:horn map} for some integer $j$, and if $j< 0$ then we can replace $j$ with zero.		
		As $\phi_{\pm 1} = T_{-1/\alpha}\circ \phi_0$, it then follows from the definitions of $\rho$ and $\chi_\pm$ that there is are integers $m_0, m_1\geq 0$ and $0\leq k < q$ such that 
		\begin{align*}
			\phi_0(z') &=T_{n-j-m_0}\circ \phi_0\circ g^{(m_0+m_1)q+k} \circ  \phi_{\pm 1}^{-1}\circ T_{j-m_1}\circ \phi_{\pm 1}(z)\\
			&= T_{n-j-m_0}\circ \phi_0\circ g^{(j+m_0)q+k}(z).
		\end{align*}
		Hence \begin{align*}
			g^{nq+k}(z) &= z'&&\text{ if }n-j-m_0\geq 0, \text{ or}\\
			g^{(j+m_0)q+k}(z) &= g^{(j+m_0-n)q}(z') &&\text{ if }n-j-m_0< 0.
		\end{align*}
		Moreover, the definition of $\rho$ implies that $g^{(j+m_0)q+k}(z)$ is entering $P_0$; if $z'$ is exiting $P_0$ it then follows that $n-j-m_0\geq 0$.
		
		If $|\text{Im}\, w|$ is sufficiently large, then the orbit of $z$ stays close to either $0$ or a non-zero fixed point $\sigma$ on $\partial P_0$. In either case, we can apply the branches of $g^{-(j+m_0)q}$ that fix these points to conclude $g^{nq+k}(z) = z'$ or $g^{k}(z) = g^{-nq}(z')$.
		If  $\pm \text{Im}\,w\gg 0$, 
		then the orbit of $z$ under $g^q$ stays close to $0$ and travels 
		from $P_0$ to $P_{\pm 2}$, so $k = q_\pm'$.
		If $\mp\text{Im}\,w\gg 0$, then instead the orbit of $z$ stays close to $\sigma$ and travels from $P_0$ to $P_0$, so we must have $k = 0$. 
	\end{proof}

	\subsection{Near-parabolic renormalization}

	A \textit{near-parabolic renormalization of $g$ relative to $f$} is a function of the form
	$$\mathcal{R}_f^\pm g:=\Exp_\pm \circ  H_\pm\circ T_{-1/\alpha}\circ(\Exp_\pm)^{-1}$$
	Proposition \ref{horn maps perturbed} implies that this map is defined on punctured neighborhoods of zero and infinity in $\hat{\mathbb{C}}$,  can be continuously extended by setting
	$\mathcal{R}_f^\pm g(0) = 0$ and $\mathcal{R}_f^\pm g(\infty) = \infty$, and we  compute the derivative  $$(\mathcal{R}_f^\pm g)'(0) = \Exp_\pm(-1/\alpha).$$

	Fixing some normalization of the Fatou coordinates for $f$ by some $(z_0, w_0)$, it follows from Proposition \ref{horn maps perturbed} that we we can choose the parabolic renormalization so that  $\RR_f^\pm g$ depends continuously and holomorphically on $g$ and converges to $\RR_\delta^\pm f$ when $g\to f$ and $-1/\alpha\to \delta$ in $\C/\Z$. 
	We saw in the last section that the parabolic renormalization $\RR_\delta^\pm f$ is uniquely defined for a fixed normalization; we have the following similar but weaker statement for near-parabolic renormalization:

	\begin{proposition}\label{prop:eventually the same renormalizations}
		Let $X\subset \Dom \mathcal{R}_0^\pm f$ be a compact set. Any two  choices of $\RR_f^\pm g$ with the same normalization agree on $X$ when $g$ is close to $f$.
	\end{proposition}

	\begin{remark}
		Just as the parabolic renormalization $\RR^\pm_\delta f$ is uniquely defined up to linear conjugacy, the near-parabolic renormalization $\RR_g^\pm f$ can be made unique. As in the parabolic case, changing the normalization conjugates the renormalization by a linear map; the main difference is that the domain of the renormalization depends on the choice of flower. As the domain always contains $0$ and $\infty$, $\RR_f^\pm g$ is uniquely defined as a pair of germs up to linear conjugacy. In the next section we will restrict to classes of maps that have a canonical choice of normalization, which makes the near-parabolic renormalizations uniquely defined.
	\end{remark}

	\subsection{Lavaurs maps}
	
	Recall that Lavaurs maps for $f$ have the form $L_\delta^\pm =\chi_\pm^f \circ T_\delta\circ \rho^f$. 
	We saw in Proposition \ref{lifing the renormalized dynamics} above that the dynamics of the relative horn map $\rho \circ \chi_\pm$ are closely related to large iterates of $g$, the relationship is even clearer for Lavaurs maps:
	
	\begin{prop}\label{prop: approximating Lavaurs maps}
		Let  $X$ be a compact subset of the $k$-th section of $\Dom(\rho^f)$ and fix some $M>0$. If $g$ is close to $f$ and $|n- \emph{Re}(1/\alpha)|\leq  M$,  then $$\chi_\pm\circ T_{n-1/\alpha}\circ \rho = g^{nq+k}$$
		on $X$.
	\end{prop}
	
	\begin{proof}
		Let $m_0$ and $m_1$ be positive integers large enough so that 
		$$f^{m_0q+k}(X)\subset P_0^f \text{ and } \phi_{\pm1}^{-1}\circ T_{m-m_0}\circ\rho^f(X)\subset P_{\pm 1}^f$$
		for $m= M$ and $m=-M$.
		It then follows from the definitions of  $\rho$ and $\chi_\pm$ that
		\begin{align*}
			\chi_\pm\circ T_{n-1/\alpha}\circ \rho &= g^{m_1q}\circ \phi_1^{-1}\circ T_{-m_1}\circ T_{n-1/\alpha}\circ T_{-m_0}\circ \phi_0\circ g^{m_0q+k}\\
			&= g^{m_1q}\circ \phi_0^{-1}\circ T_{n-m_0-m_1}\circ \phi_0\circ g^{m_0q+k}\\
			& = g^{nq+k}
		\end{align*}
		on $X$ when $g$ is close to $f$.
	\end{proof}

	Proposition \ref{prop: approximating Lavaurs maps} implies that $g^{nq+k}\to L_\delta^\pm$ on the $k$-th section of $\Dom(\rho^f)$ when $g\to f$; this is the classical theorem of parabolic implosion from \cite{douady} and \cite{Lavaurs}.

	\section{Invariant classes}\label{sec:classes}
	
	Let us fix some integer $d\geq 2$ and consider the class ${\mathcal{F}}$ of all analytic maps $f: \Dom(f)\to \hat{\mathbb{C}}$ satisfying
	\begin{enumerate}
		\item $\Dom(f)$ is an open subset of ${\mathbb{C}}$ containing $\{0, \infty\}$;
		\item $f(0) = 0,$ $f(\infty) = \infty,$ and $ f'(0) = 1$; and
		\item the restriction 
		$f: f^{-1}(\mathbb{C}^*)\to \mathbb{C}^*$
		is a branched covering whose unique critical value is $1$, and all critical points 
		are of local degree $d$.
	\end{enumerate}
	For example, $\mathcal{F}$ contains the unicritical polynomial 
	$$G(z):= 1- \left(\frac{d-z}{d}\right)^d.$$

	Following \cite{CS_satellite}, for any analytic map $f$ satisfying $f'(0) = 1$ and any $w\in \mathbb{C}$ we denote by $f\rtimes w$ the map $z\mapsto f( \Exp(w)\cdot z)$. For any class of analytic maps $\mathcal{G}$ and $X\subset \C$ we similarly denote $\mathcal{G}\rtimes X:= \{g\rtimes x: g\in \mathcal{G}, x\in X\}$.
	
	Fixing now  some $f_0\in \mathcal{F}$ and reduced $p/q\in \Q$, we set $f=f_0\rtimes p/q$. Thus $cv^f:=\Exp(p/q)$ is the unique critical value of $f$.
	\begin{proposition}
		 The map $f$ has a non-degenerate $p/q$-parabolic fixed point at zero.
	\end{proposition}
	
	\begin{proof}
		By construction $f'(0) = \Exp(p/q)$. The non-degeneracy of the parabolic fixed point follows from the uniqueness of the critical value, see for example  \cite[Theorem 10.15]{Milnor} or \cite[Lemma 4.5.2]{shishikura_1}.
	\end{proof}

	Fixing some flower  $(P_j^f)_{j\in \Z/2q\Z}$ for $f$, the \textit{parabolic basin} of $f$ is the set
	$$B^{f} = \bigcup_{m\geq 0}f^{-m}(P_{0}^{f}).$$
	The following proposition ensures that $B^f$ contains the unique critical value of $f$: 
	\begin{proposition}\label{universal cauliflower}
		The restriction of $f^q$ to $B_0^{f}$ is analytically conjugate to the restriction of $G$ to $B_0^G.$
	\end{proposition}
	\begin{proof}
		This result is classical, see for example \cite[Theorem 2.9]{lanford} ($B_0^f$ is guaranteed to be simply connected by \cite[Lemma 4.5.2]{shishikura_1}).
	\end{proof}

	Let $B_0^f$ denote the component of $B^f$ containing $P_0^f$. Relabeling the petals if necessary, we may assume that $cv^f\in B_0^f$. We normalize the Fatou coordinates for the flower by $(cv^f, 0)$ and let $H_\pm^f$ be the corresponding horn maps.
	With these choices, the horn maps and parabolic renormalizations of $f$ are uniquely defined. Moreover, the parabolic renormalizations all lie in $\FF\rtimes \C$:
	
	\begin{proposition}\label{invariance of f}
		For any $\delta\in \C$,  $\RR_\delta^\pm f\in \mathcal{F}\rtimes \pm\delta$.
	\end{proposition}
	
	\begin{proof}
		We know that $\RR_\delta^\pm f$ fixes zero and has derivative $\Exp_\pm(\delta)$ at the origin. So it suffices to show that $\RR_0 f$ is a branched covering map over $\C^*$ whose unique critical value is at one. 
		Writing $H_\pm = \rho\circ \chi_\pm$, it follows from the definition and normalization that $\rho$ is a covering map whose critical points lie in the grand orbit of $cv^f$, which is mapped to $\Z$. 
		It similarly follows from the definition that $\chi_\pm$ is a covering map over $B^f$ whose critical values all lie in the grand orbit of $cv^f$. 
		Hence $H_\pm$ is a covering map whose critical values all lie in $\Z$, which completes the proof.
	\end{proof}

	\begin{figure}
		\begin{center}
			\def\svgwidth{12.5cm}
			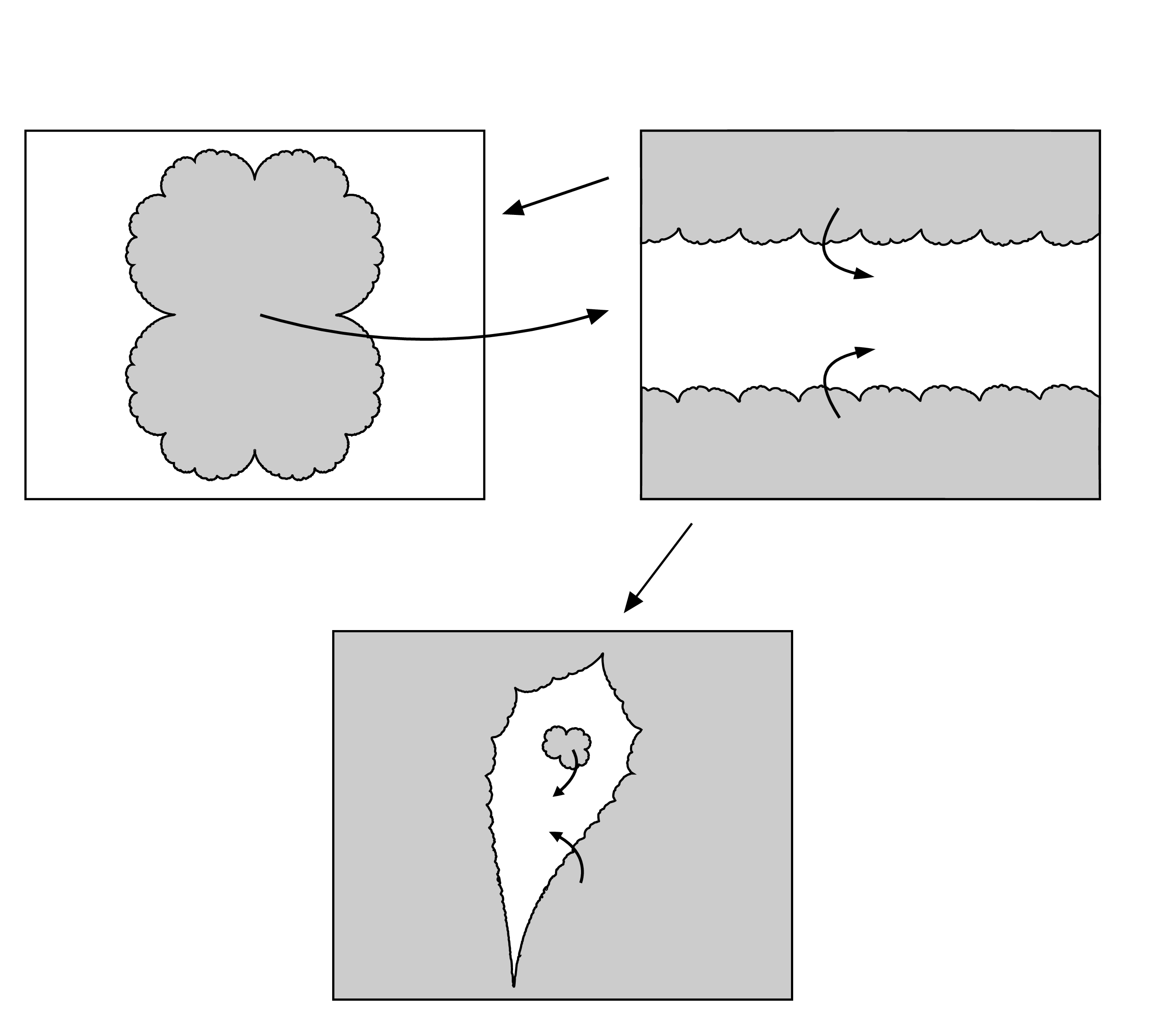
			\caption{The parabolic renormalization $\mathcal{R}_0^\pm G$ when $d = 2$.}
			\label{renormalization of G}
		\end{center}
	\end{figure}
	
	It follows from the Proposition \ref{invariance of f} that the identification $f\mapsto \RR_\delta^\pm f$ corresponds to a map $\FF\rtimes p/q\to \FF\rtimes \pm\delta$. To remove the action on the multiplier, similarly to  \cite{shishikura} we define the \textit{$p/q$-parabolic fiber renormalizations} of $f_0$ to be 
	$$\RR_{p/q, 0}^\pm f_0 :=\RR_0^\pm f = \Exp_\pm \circ H_\pm^f\circ (\Exp_\pm)^{-1}.$$
	The yields renormalization operators $\RR_{p/q, 0}^\pm:\FF\to \FF$, and the parabolic renormalization operators correspond to $(f_0, p/q)\mapsto (\RR_{p/q, 0}^\pm, \pm \delta).$

	The image of $\mathcal{R}_{p/q, 0}$ can actually be more precisely described. First we observe:
	\begin{prop}
		The components of $\Dom(\RR_0^\pm G)$ containing $0$ and $\infty$ are distinct and simply connected. 
	\end{prop}
	
	\begin{proof}
		The domain of $\RR_0^\pm G$ is exactly $\Exp_\pm\circ \chi_\pm^{-1}(B^G).$
		It follows from the maximum modulus principal that every component of $\chi_\pm^{-1}(B^f)$ is simply connected.  It follows from \cite{Orsay1} that an ``external ray" of $G$ lands at the parabolic fixed point, and consequently $\chi_\pm^{-1}(B^f)$ contains an infinite horizontal strip.
	\end{proof}
	
	\begin{remark}
		Actually one can show that the components of $\Dom(\RR_0^\pm G)$ are Jordan domains; see for example \cite{lanford}.
	\end{remark}
	
	Let $\Dom_0(\mathcal{R}_{0}^\pm G)$ be the connected component of $\Dom(\mathcal{R}_{ 0}^\pm G)$ containing $0$ (see Figure \ref{renormalization of G}); by the Riemann mapping theorem there exist unique univalent maps $$\varphi_\pm^G: \mathbb{D}\to \Dom_0(\mathcal{R}_{0}^\pm G)$$ satisfying $\varphi^G_\pm(0) = 0$ and $(\varphi_\pm^G)'(0) >0.$
	For $0 \leq \epsilon < 1$, let $\mathcal{S}_\epsilon$ be the set of univalent maps $\varphi:\mathbb{D}_{1-\epsilon}\to \mathbb{C}$ satisfying $\varphi(0) =0$ and $\varphi'(0)=1/ (\varphi_G^\pm)'(0) $. Following \cite{Cheritat22}, we define the class of maps 
	$$\mathcal{F}_\epsilon^\pm:= \{(\mathcal{R}_{0}^\pm G)\circ \varphi_G\circ \varphi^{-1}: \varphi\in \mathcal{S}_\epsilon\}.$$
	It follows from the Koebe distortion theorem that $\mathcal{S}_\epsilon$, and consequently $\mathcal{F}_\epsilon^\pm$, is compact with respect to the compact-open topology for all $0 \leq \epsilon < 1$. For any classes  $\mathcal{G}, \mathcal{G}'$ of analytic maps let us write $\mathcal{G}\sqsubset \mathcal{G}'$ if every map in $\mathcal{G}$ has a restriction in $\mathcal{G}'$. In particular, $\mathcal{F}_{\epsilon'}^\pm\sqsubset \mathcal{F}_{\epsilon}^\pm$ for any $0 \leq \epsilon' \leq \epsilon$.
	Let us also denote $\FF_\epsilon^* = \FF_\epsilon^+\cup \FF_\epsilon^-$.

	\begin{proposition}\label{weak parabolic invariance}
		For every $p/q$, $\mathcal{R}_{p/q, 0}^\pm(\mathcal{F})\sqsubset \mathcal{F}_0^\pm$.
	\end{proposition}
	
	\begin{proof}
		We only sketch the proof here (compare \cite[Theorem 13]{Cheritat22}).
		Let  $W^G$ and $W^f$ be the lifts of $\Dom_0^\pm(\RR_0^\pm G)$ and $\Dom_0^\pm(\RR_0^\pm f)$ respectively for  some $f\in \FF\rtimes p/q$. 
		An analytic isomorphism $B^G\to B^f$ as in Proposition \ref{universal cauliflower} lifts to an analytic isomorphism $\tilde\varphi: W^G\to W^f$ that commutes with $T_1$ and satisfies $H_\pm^f= H_\pm^G \circ \tilde \varphi^{-1}$.
		The map $\tilde \varphi$ descends to an analytic isomorphism $\varphi: \Dom_0^\pm(\RR_0^\pm G)\to \Dom_0^\pm(\RR_0^\pm f)$ that satisfies $\RR_0^\pm f = (\RR_0^\pm G )\circ \varphi^{-1}$; the proposition immediately follows by the Riemann mapping theorem.
	\end{proof}

	While we have only canonical parabolic renormalization for maps in $f\in \FF\rtimes \Q$, note that if $h\approx f$ is another holomorphic map with a non-degenerate $p/q$-parabolic fixed point at zero with a unique critical value $cv^h$ close to $cv^f$, we can similarly normalize the Fatou coordinates by $f$ by $(cv^h, 0).$ 
	Hence for any $p/q\in \Q$, if $\epsilon$ is sufficiently small then the parabolic fiber renormalization $\RR_{p/q, 0}$ is defined on $\FF_\epsilon^*$. One key feature of parabolic renormalization, first established by Inou and Shishikura \cite{shishikura} in the $p/q=0/1$ and $d= 2$ case, and then by Yang \cite{Yang}  in the $d= 3$ case and by Ch\'eritat \cite{Cheritat22} in the $d\geq 2$ cases (both with $p/q=0/1$), is that parabolic  fiber renormalization improves the regularity of maps in $\FF_\epsilon$:
	
	\begin{theorem}\label{parabolic invariance}
		For every  $p/q$, if $\epsilon>0$ is sufficiently small then there exists $0<\epsilon'<\epsilon$ satisfying
		$$\mathcal{R}_{p/q, 0}^\pm(\mathcal{F}_\epsilon^*)\sqsubset\mathcal{F}^\pm_{\epsilon'}.$$
	\end{theorem} 
	
	\begin{proof}
		For $p/q = 0/1$, this theorem is the main result in \cite{Cheritat22}. For general $p/q$ the same argument can be applied; we will only consider the two main steps and observe that the same reasoning applies. For details we refer the reader to \cite{Cheritat22}.
		
		Fixing some $f_0$ in $\FF^*_0$ and $f = f_0\rtimes p/q$, the first part of the argument in \cite{Cheritat22} is a contraction: showing that for any small $\epsilon>0$ there exits some $0< \epsilon'\ll \epsilon$ such that
		the restriction of $\RR_{p/q, 0}^\pm f_0 = \RR_0^\pm f$ to $\FF_{\epsilon'}^\pm$ depends only on the restriction of $f_0$ to $\FF_\epsilon^*$. Proving this fact requires showing that $B_0^f$ intersects only finitely many connected components of $X = f^{-1}(\C^*\sm \dd\D)$ and comparing the metrics:
		\begin{itemize}
			\item The hyperbolic metric on $\Dom_0(\RR_0 f).$
			\item The hyperbolic metric on $B_0^f$.
			\item  The \textit{box-Euclidean} metric on $\Dom_0(f)$, which is the lift of the flat metric on $\C^*$ by $f$.
			\item The hyperbolic metric on $\Dom_0(f)$.
		\end{itemize}
		For general $p/q$, we must consider 
		$\bigcup_{n=0}^{q-1}f^n(B_0^f)$ instead of just $B_0^f$. In \cite{Cheritat22}, showing that $B_0^f$ intersects only finitely many components of $X$ follows from studying the geometry of $(\rho^f)^{-1}(\R_{\geq -1})$; the same analysis applies for $\bigcup_{n=0}^{q-1}f^n(B_0^f)$. The comparison of the four metrics above is identical in the general $p/q$ setting. 
		
		The second part of the argument in \cite{Cheritat22} is a perturbation: showing that for any small $\epsilon$,  there is a homeomorphism $\FF_0^\pm\to \FF_\epsilon^\pm$ that does not move the  
		orbits of $f$ that induce  $\RR_0^\pm f$  very far. This part of the argument is unchanged in the general $p/q$ case.
	\end{proof}

	Fixing again some $f_0= \FF$ and $f = f_0\rtimes p/q$, 
	for some $t_0>0$ and $\alpha\in A_{t_0}$ we set $g = f_0\rtimes \mu_{p/q}^\pm(\alpha)$. 
	The map $g$ has a unique critical value $cv^g$, if $g$ is sufficiently close to $f$ then we normalize the Fatou coordinates for $g$ relative to $f$ by $(cv^g, 0)$. 
	This gives us a canonical normalization for horn maps $H_\pm^{g, f}$ of $g$ relative to $f$; while these horn maps are not unique we can choose them to depend continuously and holomorphically on $g$.
	Fixing such a choice of horn maps, we can define a \textit{$(p/q, \alpha, \pm)$-near-parabolic fiber renormalization} of $f_0$ to be 
	$$\RR_{p/q, \alpha}^\pm f_0 = (\RR_f^\pm g)\rtimes \pm1/\alpha = \Exp_\pm\circ H_\pm^{g, f}\circ (\Exp_\pm)^{-1}.$$
	Fixing holomorphic choices of horn maps near each $f\in \FF_0^*\rtimes p/q$ yields a holomorphic operators $\RR_{p/q, \alpha}^\pm$ defined on $\FF_\epsilon^\pm$ for all sufficiently small $\epsilon\geq 0$ and $\alpha\in A_{t_0}$.
	Note that while  the near-parabolic fiber renormalization operators are not uniquely defined, they are all related by Proposition \ref{prop:eventually the same renormalizations}.

	The improvement of regularity for the parabolic fiber renormalization operators implies the same for the near-parabolic fiber renormalizations:
   	 \begin{theorem}\label{near-parabolic invariance}
   	 	For any $p/q$, $t_0$, and choice of $\mathcal{R}_{p/q, \alpha}^\pm$, if $\epsilon>0$ is sufficiently small then there exist $r>0$ and $0 < \epsilon' < \epsilon$ such that 
   	 	$$\mathcal{R}_{p/q, \alpha}(\mathcal{F}_\epsilon^*)\subset \mathcal{F}_{\epsilon'}^\pm$$
   	 	for all $\alpha \in A_{t_0}\cap \D_r$.
   	 \end{theorem} 
   	 
   	 \begin{proof}
   	 	This follows from Theorem \ref{parabolic invariance} and the convergence $\RR_{p/q, \alpha}^\pm\to \RR_{p/q, 0}^\pm$ when $\alpha\to 0$.
   	 \end{proof}
   	 
	The following fact from \cite{shishikura} shows that the parabolic and near-parabolic fiber renormalization operators are contracting on $\mathcal{F}_\epsilon^*$:
	\begin{proposition}\label{contraction}
		For any $0 < \epsilon' < \epsilon<1$ there exists complete metrics $\mathfrak{d}^\pm$ on $\mathcal{F}_{\epsilon'}^\pm$ such that if $\mathcal{R}^\pm: \mathcal{F}_\epsilon^\pm\to \mathcal{F}_{\epsilon'}^\pm$ are holomorphic operators, then
		$$\mathfrak{d}^\pm(\mathcal{R}^\pm(f_1), \mathcal{R}^\pm(f_2))< \frac{1-\epsilon}{1-\epsilon'}\mathfrak{d}^\pm(f_1, f_2)$$
		for all $f_1, f_2\in \mathcal{F}_{\epsilon'}^\pm$.
		Moreover, convergence in this metric implies convergence in the compact-open topology.
	\end{proposition}
	\begin{proof}
		See \cite[Main Theorem 2]{shishikura}.
	\end{proof}
	
	As the map $x\mapsto - 1/(\mu_{p/q}^\pm)^{-1}(x)$ is expanding near $x = p/q$ and Proposition \ref{contraction} implies that $f_0\mapsto \RR_{p/q, \alpha}^\pm f_0$ is contracting, it follows that 
	the near-parabolic renormalization operators	
	$$
		\RR_{p/q}^\pm: (f_0, \mu_{p/q}^\pm(\alpha))\mapsto (\RR_{p/q, \alpha}^\pm f_0, \mp1/\alpha)
	$$
	are hyperbolic; for a more detailed discussion of hyperbolicity of near-parabolic renormalization operators  we refer the reader to \cite{shishikura} and \cite{CS_satellite}.
	
	\section{Relating renormalizations}\label{sec:repeated}
	
	Now let us consider a map that has a fixed point with multiplier close to two distinct roots of unity. 
	These different roots of unity may give two different parabolic renormalizations,  in this section we show how to relate these renormalizations.

	Let $\kappa$, $\kappa_0$, and $\kappa_1$ be modified continued fractions such that 
	\begin{equation*}
		\mu_\kappa(z) = \mu_{\kappa_0}\left(\frac{1}{n+ \epsilon\mu_{\kappa_1}(z)}\right).
	\end{equation*}
	for some $n\geq 2$ and $\epsilon = \pm 1$.
	Let $p_0/q_0$, $p_1/q_1$, and $p/q$ be the rational numbers with modified continued fraction expansions $\kappa_0$, $\kappa_1$, and $\kappa$ respectively. 
	Thus there are some $s_0, s_1, s\in \{\pm\}$ such that $\mu_{\kappa_0}= \mu_{p_0/q_0}^{s_0}$,
	$\mu_{\kappa_1}= \mu_{p_1/q_1}^{s_1}$, 
	and $\mu_{\kappa}= \mu_{p/q}^{s}$.
	Identifying  $\pm$ with $\pm 1$, it follows from the construction that $s = -\epsilon s_0s_1$. 
	For the rest of this section we keep $\kappa_0$ and $\kappa_1$ fixed, allowing $p/q$ to vary by sending $n\to \infty$ and choosing $\epsilon$ so that $s = s_1$.
	Note in particular that $p/q\to p_0/q_0$ when $n\to \infty$.
	Using Proposition \ref{prop:mobius form}, we can compute 
	\begin{equation}\label{eq: q}
		q = nq_0q_1+\epsilon p_1q_0+ q_0'q_1,
	\end{equation}
	where $0\leq q_0'< q_1$ is the integer satisfying $p_0q_0' + s_0 \equiv 0 \mod q_0$.

	Let $f_0$ be an analytic map with a non-degenerate $p_0/q_0$-parabolic fixed point at zero and let $(P_j^{f_0}, \phi_j^{f_0})_{j\in \Z/2q_0\Z}$ be a parabolic flower for $f_0$.
	We set $\delta_0 = s_0p_1/q_1$ and $f_1 = \RR^{s_0}_{\delta_0}f_0$.
	Thus $f_1$ has a $p_1/q_1$-parabolic fixed point at zero; we assume that this parabolic point is also non-degenerate.
	 Let $(P_j^{f_1}, \phi_j^{f_0})_{j\in \Z/2q_1\Z}$ be a parabolic flower for $f_1$.
	
	Let $g_0$ be an analytic map with a $p/q$-parabolic fixed point at zero
	and let $(P_j^{g_0}, \phi_j^{g_0})_{j\in \Z/2q\Z}$ be a parabolic flower for $g_0$.
	If $g_0$ is sufficiently 	
	close to $f_0$, then a near-parabolic renormalization $g_1=\RR_{f_0}^{s_0}g_0$ is defined and close to $f_1$. As $$ g_1'(0)= \Exp_{s_0}(-\epsilon p_1/q_1) = \Exp(p_1/q_1),$$  $g_1$ also has a non-degenerate $p_1/q_1$-parabolic fixed point at zero when $g_1$ is close to $f_1$.  Let $(P_j^{g_1}, \phi_j^{g_1})_{j\in \Z/2q_1\Z}$ be a parabolic flower for $g_1$.
	
	Fixing some $t_0>0$, let $h_0$ be an analytic map such that
	$$h_0(0) = 0\text{ and }h_0'(0) =\Exp\circ \mu_{\kappa}(\alpha)$$ for some $\alpha\in A_{t_0}$. 
	We set $\alpha_1 = \alpha$ and $\alpha_0 = \frac{1}{n+\epsilon \mu_{\kappa_1}(\alpha)}$, so $\mu_{\kappa}(\alpha) = \mu_{\kappa_0}(\alpha_0).$
	If $h_0$ is close to $g_0$, then there is a near-parabolic flower   $(P_j^{h_0, g_0}, \phi_j^{h_0, g_0})_{j\in \Z/2q\Z}$ near zero for $h_0$ relative to $g_0$ with corresponding near-parabolic renormalization
	$\RR_{g_0}h_0$.
	As $g_0$ is close to $f_0$,
	there is also a near-parabolic flower $(P_j^{h_0, f_0}, \phi_j^{h_0, f_0})_{j\in \Z/2q_0\Z}$ near zero for $h_0$ relative to $f_0$ with corresponding near-parabolic renormalization $h_1= \RR_{f_0}h_0$.
	Note that 
	$$h_1'(0) = \Exp_{s_0}(-\epsilon\mu_{\kappa_1}(\alpha)) = \Exp\circ \mu_{p_1/q_1}^{s_1}(\alpha).$$
	As $h_1$ is close to $g_1$, which is close to $f_1$, 
	there is a near-parabolic flower $(P_j^{h_1, f_1}, \phi_j^{h_1, f_1})_{j\in \Z/2q_1\Z}$ near zero for $h_1$ relative to $f_1$ with corresponding near-parabolic renormalization $h_2= \RR_{f_1}h_1$.
	
	We observe that we can directly relate the flowers for $h_0$ and $h_1$ (see Figure \ref{fig:lifts}):
	
		\begin{figure}
		\begin{center}
			\def\svgwidth{12.5cm}
\begingroup%
  \makeatletter%
  \providecommand\color[2][]{%
    \errmessage{(Inkscape) Color is used for the text in Inkscape, but the package 'color.sty' is not loaded}%
    \renewcommand\color[2][]{}%
  }%
  \providecommand\transparent[1]{%
    \errmessage{(Inkscape) Transparency is used (non-zero) for the text in Inkscape, but the package 'transparent.sty' is not loaded}%
    \renewcommand\transparent[1]{}%
  }%
  \providecommand\rotatebox[2]{#2}%
  \newcommand*\fsize{\dimexpr\f@size pt\relax}%
  \newcommand*\lineheight[1]{\fontsize{\fsize}{#1\fsize}\selectfont}%
  \ifx\svgwidth\undefined%
    \setlength{\unitlength}{1092.85997286bp}%
    \ifx\svgscale\undefined%
      \relax%
    \else%
      \setlength{\unitlength}{\unitlength * \real{\svgscale}}%
    \fi%
  \else%
    \setlength{\unitlength}{\svgwidth}%
  \fi%
  \global\let\svgwidth\undefined%
  \global\let\svgscale\undefined%
  \makeatother%
  \begin{picture}(1,0.82541135)%
    \lineheight{1}%
    \setlength\tabcolsep{0pt}%
    \put(0,0){\includegraphics[width=\unitlength]{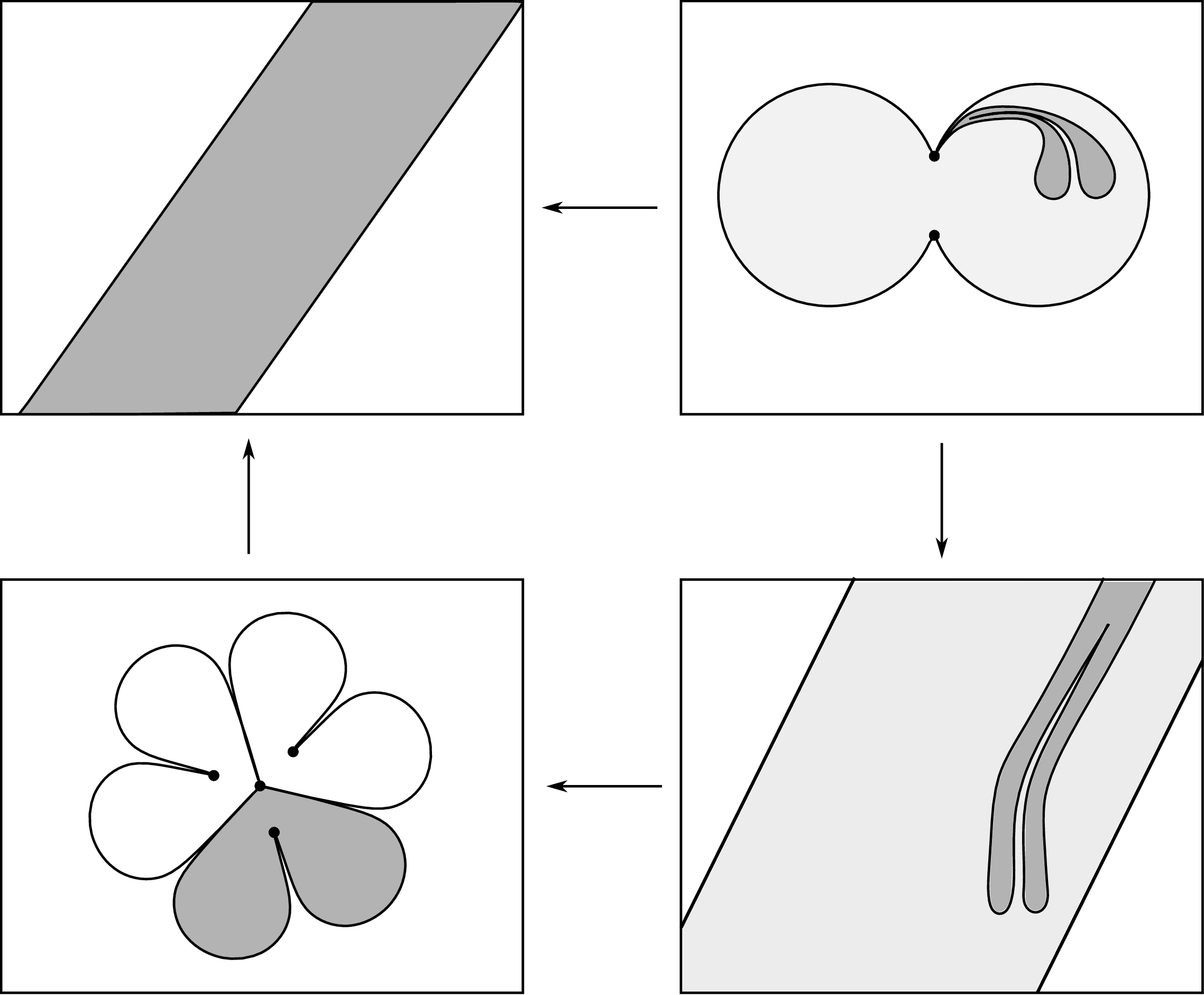}}%
    \put(0.23353261,0.40572011){\makebox(0,0)[lt]{\lineheight{1.25}\smash{\begin{tabular}[t]{l}$\phi_0^{h_1, f_1}$\end{tabular}}}}%
    \put(0.81109957,0.40572011){\makebox(0,0)[lt]{\lineheight{1.25}\smash{\begin{tabular}[t]{l}$\phi_0^{h_0, f_0}$\end{tabular}}}}%
    \put(0.46714289,0.67379753){\makebox(0,0)[lt]{\lineheight{1.25}\smash{\begin{tabular}[t]{l}${\phi}_0^{h_0, g_0}$\end{tabular}}}}%
    \put(0.46334174,0.19246748){\makebox(0,0)[lt]{\lineheight{1.25}\smash{\begin{tabular}[t]{l}$\Exp_{s_0}$\end{tabular}}}}%
    \put(0.3264579,0.04311435){\makebox(0,0)[lt]{\lineheight{1.25}\smash{\begin{tabular}[t]{l}$P_0^{h_1, f_1}$\end{tabular}}}}%
    \put(0.84737643,0.62428889){\makebox(0,0)[lt]{\lineheight{1.25}\smash{\begin{tabular}[t]{l}$P_0^{h_0, g_0}$\end{tabular}}}}%
    \put(0.60831207,0.76547894){\makebox(0,0)[lt]{\lineheight{1.25}\smash{\begin{tabular}[t]{l}$P_0^{h_0, f_0}$\end{tabular}}}}%
  \end{picture}%
\endgroup%

			\caption{The relationship between flowers as in Proposition \ref{prop:comparing petals}. Note that the petals should be spiraling around the fixed points; we omit this detail for clarity of the image.}
			\label{fig:lifts}
		\end{center}
	\end{figure}
	
	\begin{prop}\label{prop:comparing petals}
		We can choose the flowers  so that $\Exp_{s_0}\circ \phi_0^{h_0, f_0}$ maps $P_0^{h_0, g_0}$ to $P_0^{h_1, f_1}$, conjugates $h_0^q$ to $h_1^{q_1}$ there, and satisfies $$\phi_0^{h_0, g_0} = \phi_1^{h_1, f_1}\circ \Exp_{s_0}\circ \phi_0^{h_0, f_0}.$$
	\end{prop}
	\begin{proof}
		Setting $n_0= n$, note that we have that $\text{Im}\alpha_0/\text{Re}\alpha_0\to 0$ when $n\to \infty$. 
		It therefore follows from Proposition \ref{prop:tilt} that we can choose the tilt of 
		the flowers fo that there is some set $X$ of exiting points in $P_0^{h_0, f_0}$ such that $Y = \phi_0^{h_0, f_0}(X)$ is mapped univalently to $P_0^{h_1, f_1}$ by $\Exp_{s_0}$.
		
		Let $H^{h_0, f_0}$ be the horn map for $h_0$ relative to $f_0$ corresponding to the choice of flower, so $H^{h_0, f_0}(w) - w\to 0$ when $\text{Im}\,w\to s_0\infty$.
		As $h_1^{q_1}$ maps most of $P_{0}^{h_1, f_1}$ to itself and zero is on the boundary of $P_0^{h_1, f_1}$, it follows that there 
		$$w\mapsto(H^{h_0, f_0}\circ T_{n_0-1/\alpha_0})^{q_1}\circ T_{\epsilon p_1}(w) \approx w$$
		maps most of $Y$ to itself and is conjugated to $h_1^{q_1}$ by $\Exp_{s_0}$.
		As $X$ is exiting $P_0^{h_0, f_0}$, it then follows from Proposition \ref{lifing the renormalized dynamics} and \eqref{eq: q} that $h_0^q$ maps most of $X$ to itself and is conjugated to $h_1^{q_1}$ by $\Exp_{s_0}\circ \phi_0^{h_0, f_0}.$
		Thus $X$ is a petal for $h_0^q$ with Fatou coordinate $\Exp_{s_0}\circ \phi_0^{h_0, f_0}$. 
		
		Setting $X_{-2pj}= h^{-j}(X)$ and $\varphi_{-2pj} = \Exp_{s_0}\circ \phi_0^{h_0, f_0}\circ h^j$ for all $0\leq j < q$,  
		$X_{j+s} = X_j$ and $\varphi_{j+s} = \varphi_j$ for all even $j$, the collection $(X_j, \varphi_j)$ is a near-parabolic flower for $h_0$ relative to $g_0$. 
		In particular, we can choose $(P_j^{h_0, g_0}, \phi_j^{h_0, g_0})= (X_j, \varphi_j)$.
	\end{proof}

	As a consequence of Proposition \ref{prop:comparing petals}, we can relate the near-parabolic renormalizations:
	
	\begin{theorem}\label{thm:comparison}
		For any compact $X\in \Dom(f_2)$, if $g_0$ is close to $f_0$ and $h_0$ is close  to $g_0$, then we can choose the  near-parabolic renormalizations so that 
		$h_2= \RR_{g_0}^sh_0$ on $X$.
	\end{theorem}

	\begin{proof}
		Let $H^{f_0}_{s_0}$ and $H^{f_1}_{s_1}$ be horn maps for $f_0$ and $f_1$ respectively. 
		Let $H^{h_0, g_0}$, $H^{h_0, f_0}$,  and $H^{h_1, f_1}$ be horn maps for $h_0$ relative to $g_0$, $h_0$ relative to $f_0$, and $h_1$ relative to $f_1$ respectively. As $\Exp_{s_1}= \Exp_{s}$ semi-conjugates $H^{h_1, f_1}\circ T_{-1/\alpha}$ and $H^{h_0, g_0}\circ T_{-1/\alpha}$ to $h_2$ and $\RR_{g_0}h_0$ respectively, the theorem is a consequence of the following  claim: for any compact $X\subset \Dom(H^{f_1}_{s_1})$, we can choose the horn maps so that 
		$H^{h_1, f_1}-H^{h_0, g_0}$ maps $X$ into $\Z$.
		
		Let $X_1\subset P_{s_1}^{f_1}$ be a compact set such that $f_1^{N_1}(X_1)\subset P_0^{f_1}$ 
		for some integer $N_1\geq 0$. 
		Thus $X_1\subset P_{s_1}^{h_1, f_1}$ and $h_1^{N_1}(X_1)$ is entering $P_0^{h_1, f_1}$ when $h_1$ is close to $f_1$. 
		Assuming the flowers are chosen as in Proposition \ref{prop:comparing petals}, let $X_0$ be the component of $(\Exp_{s_0}\circ \phi_0^{h_0, f_0})^{-1}(X_0)$ contained in $P_0^{h_0, g_0}$.
		Note that it follows from the proof of Proposition \ref{prop:comparing petals} that we can choose the petals so that $P_0^{h_0, g_0}$ is exiting $P_0^{h_0, f_0}$.

		We set $Y_0= \phi_0^{h_0, f_0}(X_0)$, so $\Exp_{s_0}(Y_0)= X_1$. 
		Setting $n_0 = n$, and 
		choosing some integers $j_m$ for all $0\leq m < m_1$, we inductively define $$Y_{m+1} = H^{h_0, f_0}\circ T_{n_0+j_m-1/\alpha_0} (Y_m).$$
		When $h_1$ is sufficiently close to $f_1$ and $h_0$ is sufficiently close to $f_0$, we can choose the integers $j_m$, depending only on $X_1$, so that each $(\phi_0^{h_0, f_0})^{-1}(Y_m)$ is defined and exiting $P_0^{h_0, f_0}$. 
		Moreover, as $h_1^{N_1}(X_1)$ is entering $P_0^{h_1, f_1}$, we can choose the integers so that $(\phi_0^{h_0, f_0})^{-1}(Y_{N_1})$ is entering $P_0^{h_0, g_0}$.
		It then follows from Proposition \ref{lifing the renormalized dynamics} that when $\alpha$ is sufficiently small there is some $$N_0\leq \sum_{m=0}^{m_1-1}(n_0+j_m+1)q_0< \text{Re}(1/3\alpha)q$$
		such that $h_0^{N_0}(X_0)$ is entering $P_0^{h_0, g_0}$
		and $\Exp_{s_0}\circ \phi_0^{h_0, f_0}\circ h_0^{N_0} = h_1^{N_1}$ on $X_0$. 
		Hence for any $z_0\in X_0$, $z_1 = \Exp_s\circ \phi_0^{h_0, f_0}(z_0)$, and $w= \phi_0^{h_0, g_0}(z_0) = \phi_0^{h_1, f_1}(z_1)$, we have that 
		$$H^{h_0, g_0}(w) \equiv \phi_0^{h_0, g_0}\circ h_0^{N_0}(z_0) = \phi_0^{h_1, f_1}\circ h_1^{N_1}(z_1)\equiv H^{h_1, f_1}(w)$$
		modulo $\Z$. For any compact set $X\subset \Dom(H^{f_1}_{s_1})$, we can choose $X_1$ so that $\phi_0^{h_1, f_1}(X_1)+m$ contains $X$ for some integer $m$, which completes the proof.
	\end{proof}
	
	In the context of the fiber renormalizations in the previous section, Theorem \ref{thm:comparison} has the following interpretation:
	\begin{cor}\label{cor:comparison}
		For any small $\epsilon>0$,  when $n$ is sufficiently large we can choose the near-parabolic renormalization operators so that
		$$\RR_{p_1/q_1, \alpha_1}^{s_1}\circ \RR_{p_0/q_0, \alpha_0}^{s_0} = \RR_{p/q, \alpha}^s$$
		on $\FF_\epsilon^*$.
	\end{cor}
	
	\begin{proof}
		Choosing $f_0$, $g_0$, and $h_0$ above so that $f_0 = f\rtimes p_0/q_0$,  $g_0 = f\rtimes p/q$, and $h_0 = f\rtimes \mu_\kappa(\alpha)$ for some $f\in \FF_\epsilon^*$, it follows from the definitions that
		we can choose the operators so that 
		$\RR_{p/q, \alpha}^{s} f= (\RR_{g_0}h_0)\rtimes s/\alpha$ and 
		$$\RR_{p_1/q_1, \alpha_1}^{s_1}\circ\RR_{p_0/q_0, \alpha_0}^{s_0} f = \RR_{p_1/q_1, \alpha_1}^{s_1} (h_1\rtimes -\mu_{p_1/q_1}^{s_1}(\alpha_1)) = (\RR_{h_1\rtimes (p_1/q_1-\mu_{p_1/q_1}^{s_1}(\alpha_1))}h_1)\rtimes s_1/\alpha_1.$$
		As $h_1\rtimes (p_1/q_1-\mu_{p_1/q_1}^{s_1}(\alpha_1))$ has a $p_1/q_1$-parabolic fixed point at zero and converges to $f_1$ when $n\to \infty$ and $\alpha\to 0$, it follows from Proposition \ref{prop:near-parabolic independence} that we can choose the operators so that 
		$$\RR_{p_1/q_1, \alpha_1}^{s_1}\circ\RR_{p_0/q_0, \alpha_0}^{s_0} f =( \RR_{f_1}h_1)\rtimes s_1/\alpha_1.$$
		As $\RR_{f_1}^{s_1}h_1 =  \RR_{g_0}^{s_0}h_0$ by Theorem  \ref{thm:comparison} and $s_1/\alpha_1= s/\alpha$, we are done.
	\end{proof}
	
	Using Corollary \ref{cor:comparison}, we can see that the constants in Theorem \ref{near-parabolic invariance} depend only on the length of the continued fraction.
	More precisely, we have:
	\begin{cor}\label{uniform renormalization}
		Fix some $n\geq 0$. For any $t_0$ and small $\epsilon>0$, there exists $r>0$ and $0 < \epsilon' < \epsilon$ such that: for any $p/q\in \Q_n$ and choice of $\RR_{p/q, \alpha}$, we have 
		$$\mathcal{R}_{p/q, \alpha}^\pm(\mathcal{F}_\epsilon^*)\subset \mathcal{F}_{\epsilon'}^\pm$$
		for all $\alpha \in A_{t_0}\cap \D_r$.
	\end{cor}

	A natural question which arises from Corollary \ref{uniform renormalization} is whether the constants can be chosen uniformly over all $p/q\in \Q$, without bounding the length of the modified continued fraction. We might suspect that the answer is yes, but an answer require understanding the behavior of  $\mathcal{R}_{p/q, 0}^\pm$ when $p/q$ tends toward an irrational number.

\appendix

	\section{The geometry of petals}

	In this appendix we consider the geometry of  flowers and near-parabolic flowers and provide  proofs that were omitted for some statements in Sections \ref{sec:parabolic} and \ref{sec:near-parabolic}. We fix some $p/q\in \Q$; when $p/q= 0/1$ all the statements here are given (with minor modification) in \cite{shishikura_1}, with some additional observations in \cite{positive_area}. We will not completely recreate the classical analysis, instead focusing on the necessary changes for the general setting. For more details, we refer the reader to \cite{shishikura_1}. Similar modifications  to generalize to the  $p/q$ setting are suggested in \cite[\S 7]{shishikura_boundary}; we provide more details than \cite[\S 7]{shishikura_boundary} to avoid the error discussed in Remark \ref{rem:error}.

	A \textit{sector} in $\C$ is a component of the complement of two intersecting lines. For sectors $S_1\subset S_2$, we will say that $S_1$ is \textit{well-inside} $S_2$ if none of the lines which form the boundaries of $S_1$ and $S_2$ are parallel. For any open set $U\subset \C$, we will say that a sector $S_1$ is well-inside $U$ if it is well-inside a sector $S_2\subset U$.

	For $s>0$ and $z_1, z_2\in \C$ satisfying $\text{Re}(z_1-z_2)\leq s\cdot |\text{Im}(z_1-z_2)|$, we define
	$$\QQ(z_1, z_2, s):= \{w\in \C: \text{Re}\,z_1- s\cdot|\text{Im}(w-z_1)|< \text{Re}\,w< \text{Re}\,z_2+s\cdot|\text{Im}(w-z_2)|\}.$$
	See for example Figure \ref{petal def fig}.
	We will also allow $z_1=-\infty$ or $z_2=+\infty$, in these cases we ignore the inequalities containing $z_1$ or $z_2$ respectively. 

	Our main tool for conjugating $f^q$ to $z\mapsto z+1$ is the following result from \cite{shishikura_1}:
	\begin{prop}\label{fatou coordinates}
		Fix some $s>0$ and let $f$ be a holomorphic function defined on some ${\QQ}={\QQ}(b_1, b_2, s)$ with $\emph{Re}\,b_2>\emph{Re}\,b_1+2.$ If $\epsilon>0$ is sufficiently small and if 
		$$|f(w)-(w+1)|< \epsilon\text{ and }|f'(w)-1|< \epsilon$$
		for all $w\in {\QQ}$, then:
		\begin{enumerate}
			\item $g$ is univalent on ${\QQ}$.
			\item Fix some  $x\in \mathbb{R}$ satisfying $\emph{Re}\, b_1< x< \emph{Re}\, b_2-1-\epsilon$ and set $\ell = \{x+iy:y\in \mathbb{R}\}$. For any $w\in \mathcal{Q}$ there exists a unique integer $n$ such that $f^n(w)$ belongs to the strip bounded by $\ell$ and $f(\ell).$
			\item There exists a univalent function $\Phi:{\QQ}\to \mathbb{C}$, unique up to post-composition by a translation, such that
			$$\Phi\circ f(w)= \Phi(w)+1$$
			wherever both sides of the equation are defined.
			\item Fix some point $w_0\in \mathcal{Q}$. If we normalize by $\Phi(w_0)=0$, then $\Phi$ depends continuously and holomorphically on $f$. 
			\item If $$f(w) = w+1+\frac{c}{w}+O\left(\frac{1}{|w|^{1+m}}\right)$$
			near $w=\infty$ for some $m>0$, then for any sector $S$ well-inside $\QQ$ there is a constant $C$ such that 
			$$\Phi(w) = w+c\log w+C+o(1)$$
			when $w\to \infty$ in $S$.
			\item For any $0< s'< s$, there is some $R>0$ depending only on $s$ and $s'$ such that if $\epsilon$ is sufficiently small then 
			 $$\QQ(\Phi(b_1+R), \Phi(b_2-R), s')\subset \Phi(\QQ).$$
		\end{enumerate}
	\end{prop}
	\begin{proof}
		Parts $(1)$-$(4)$ are exactly \cite[Proposition 2.5.2]{shishikura_1} with $\epsilon= 1/4$ and $s = 1$; the argument in general is identical.
		Part $(5)$  follows immediately from \cite[Proposition 2.6.2]{shishikura_1}, and part $(6)$ is exactly \cite[Lemma 3.5.1]{shishikura_1}
		for a particular choice of $\epsilon$, $s$, and $s'$; the general argument is identical.
	\end{proof}

	\subsection{Parabolic flowers}
	Let $f$ be an analytic function with a non-degenerate $p/q$-parabolic fixed point and defined in a neighborhood $V$ of zero.
	Up to analytic conjugacy, $f$ has the form 
	$$f(z)= e^{2\pi i p/q}z(1+z^{ q}+\beta z^{2 q}+O(z^{2 q + 1}))$$
	near $z=0$, 
	where $\beta\in \C$ is the \textit{formal invariant} of $f$ at $0$; see for example \cite[Appendix]{parabolic_PLY}. For $n\geq 1$, we can compute
	$$f^n(z) = e^{2\pi inp/q}z\left(1+nz^q+\beta_nz^{2q}+O(z^{2q+1})\right)$$
	near $z=0$ for some $\beta_n$.
	For every $n\geq 1$, conjugating by $\psi(z) :=q^2z^q$ yields the multi-valued map $\tilde{f}^n= \psi\circ f^n\circ \psi^{-1}$, every branch of which satisfies
	$$\tilde{f}^n(z) =  z\left(1+\frac{n}{q}z+\tilde{\beta}_nz^2+O(z^{2+1/q})\right)$$
	for some $\tilde{\beta}_n$ depending only on $\beta$ and $n$. 
	Note that we abuse notation a bit: the superscript in $\tilde{f}^n$ does not indicate an iterate, for example  $\tilde{f}^2\neq \tilde{f}^1\circ \tilde{f}^1$, but each branch of $\tilde{f}^2$ is a branch of $\tilde{f}^1\circ \tilde{f}^1$. We will use this convention for all the ``iterates" of multi-valued maps in this section.
	For $\tau(z) =-1/z$,
	we can conjugate $\tilde{f}^n$ to the multi-valued map ${F}^n= \tau\circ \tilde{f}^n\circ \tau^{-1}$, every branch of which satisfies 
	$${F}^n(w) = w+\frac{n}{q}+ \frac{{B}_n}{w}+ O\left(\frac{1}{w^{1+1/q}}\right)$$
	near $w=\infty$ for some ${B}_n$ that depends only on $\beta$ and $n$.
	As for $\tilde{f}^n$, the superscript in $F^n$ is not an iterate.
	Choosing some $t_0>0$ and $\epsilon>0$, let $\xi>0$ be large enough so that for for all $1\leq n \leq q$, we have
	$$|{F}^n(w) - (z-n/q)|< \epsilon/2 \text{ and }|({F}^n)'(z) - 1|< \epsilon/2$$
	for all $w$ in the sets
	$$\QQ^0 = \QQ(\xi, +\infty, t_0) \text{ and }\QQ^1= \QQ(-\infty, -\xi, t_0)$$
	and every branch of ${F}^n$.
	
	Fixing some continuous branch $\psi_0^{-1}$ of $\psi^{-1}$ defined on $\tau(\QQ^0)$, let $\psi_{-1}^{-1}$ and $\psi_{1}^{-1}$ be the continuous branches of $\psi^{-1}$ on $\tau(\QQ^1)$ 
	such that $\psi^{-1}_0\circ \tau$ agrees with $\psi_{-1}^{-1}\circ \tau$ and $\psi_{1}^{-1}\circ \tau$ on the upper and lower components of $\QQ^0\cap \QQ^1$ respectively. 
	For any $j\in \Z$, we inductively define $\psi_{j+2}^{-1} = e^{2\pi i /q}\cdot \psi_j^{-1}$. 
	Using this labeling of the branches of $\psi^{-1}$, we label the branches of ${F}^q$ by 
%
%
%
 $${F}^q_j= \tau^{-1}\circ \psi\circ f^q\circ \psi_j^{-1}\circ \tau: \QQ^{[j]}\to \C.$$
%
	Proposition \ref{fatou coordinates} implies that
	there is a univalent map $\Phi_j: \QQ^{[j]}\to \C$  conjugating ${F}^q_j$ to $T_1$
	for each $j\in \Z/2q\Z$.
	Fixing an attracting or repelling strip $S_j\subset \Phi_j(\QQ^{[j]})$ with tilt $t$  for each $j$, which is guaranteed to exist  by part (6) of Proposition \ref{fatou coordinates}, we set (see Figure \ref{fig flower parabolic})
	$$P_j = \psi_j^{-1}\circ \tau\circ \Phi_j^{-1}(S_j) \text{ and }\phi_j=\Phi_j\circ\tau^{-1}\circ \psi.$$
	With these definitions, the following proposition immediately implies Theorem \ref{flower}:
	
		\begin{figure}
		\begin{center}
			\def\svgwidth{4in}
\begingroup%
  \makeatletter%
  \providecommand\color[2][]{%
    \errmessage{(Inkscape) Color is used for the text in Inkscape, but the package 'color.sty' is not loaded}%
    \renewcommand\color[2][]{}%
  }%
  \providecommand\transparent[1]{%
    \errmessage{(Inkscape) Transparency is used (non-zero) for the text in Inkscape, but the package 'transparent.sty' is not loaded}%
    \renewcommand\transparent[1]{}%
  }%
  \providecommand\rotatebox[2]{#2}%
  \newcommand*\fsize{\dimexpr\f@size pt\relax}%
  \newcommand*\lineheight[1]{\fontsize{\fsize}{#1\fsize}\selectfont}%
  \ifx\svgwidth\undefined%
    \setlength{\unitlength}{407.51781554bp}%
    \ifx\svgscale\undefined%
      \relax%
    \else%
      \setlength{\unitlength}{\unitlength * \real{\svgscale}}%
    \fi%
  \else%
    \setlength{\unitlength}{\svgwidth}%
  \fi%
  \global\let\svgwidth\undefined%
  \global\let\svgscale\undefined%
  \makeatother%
  \begin{picture}(1,0.85073703)%
    \lineheight{1}%
    \setlength\tabcolsep{0pt}%
    \put(0,0){\includegraphics[width=\unitlength]{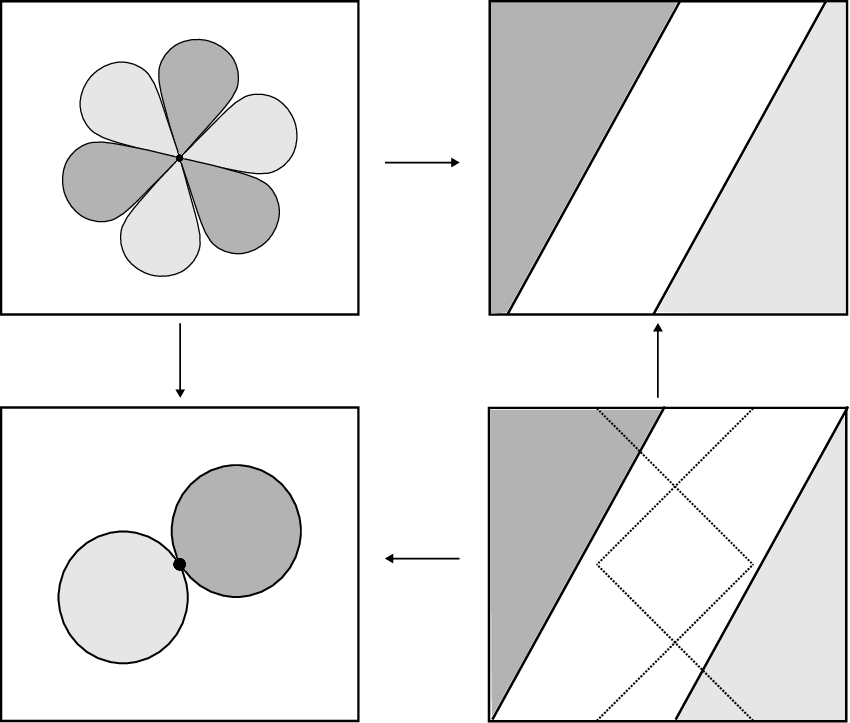}}%
    \put(0.48302221,0.21418154){\makebox(0,0)[lt]{\lineheight{1.25}\smash{\begin{tabular}[t]{l}$\tau$\end{tabular}}}}%
    \put(0.229368,0.42119756){\makebox(0,0)[lt]{\lineheight{1.25}\smash{\begin{tabular}[t]{l}$\psi$\end{tabular}}}}%
    \put(0.78798146,0.42119756){\makebox(0,0)[lt]{\lineheight{1.25}\smash{\begin{tabular}[t]{l}$\Phi_j$\end{tabular}}}}%
    \put(0.68217832,0.0973596){\makebox(0,0)[lt]{\lineheight{1.25}\smash{\begin{tabular}[t]{l}$\QQ^1$\end{tabular}}}}%
    \put(0.84630336,0.24343753){\makebox(0,0)[lt]{\lineheight{1.25}\smash{\begin{tabular}[t]{l}$\QQ^0$\end{tabular}}}}%
    \put(0.4686107,0.68176599){\makebox(0,0)[lt]{\lineheight{1.25}\smash{\begin{tabular}[t]{l}$\phi_{j}$\end{tabular}}}}%
  \end{picture}%
\endgroup%

			\caption{Construction of a parabolic flower.}
			\label{fig flower parabolic}
		\end{center}
	\end{figure}

	\begin{prop}
		We can choose $(P_j, \phi_j)_{j\in \Z/2q\Z}$ so that it is a $p/q$-parabolic flower for $f$ with tilt $t$ inside $V$.
	\end{prop}
	
	\begin{proof}
		It follows immediately from the definition that each $P_j\subset V$ is a Jordan domain with zero on its boundary.
		As each $S_j$ has the same tilt, it follows from the estimate for  $\Phi_j$ in part (5) of Proposition \ref{fatou coordinates} that, by choosing smaller strips $S_j$ if necessary, the closures of any two petals intersect only at zero. 
		Our definition of the branches of $\psi^{-1}$ implies that the circular ordering of the sets $P_j$ around zero is given by the ordering of $\Z/2q\Z$. 
		This proves the first condition in the definition of a parabolic flower; the other conditions follow easily from the construction.
	\end{proof}
	
	\begin{proof}[Proof of Theorem \ref{flower}]
		The collection $(P_j, \phi_j)$ is a desired flower.
	\end{proof}

	We can also now verify some propositions from Section \ref{sec:parabolic}:

	\begin{proof}[Proof of Proposition \ref{prop:petals}]
		If $(\tilde{P}_j, \tilde{\phi_j})$ is another flower, then, shrinking the flower if necessary, $\tau^{-1}\circ \psi$ maps each $\tilde{P}_j$ into $\QQ^{0}\cup \QQ^{1}$ for some $j$. It follows that, after analytically extending if necessary, there is some even $s$ so that $\Phi_{j+s}\circ \tau^{-1}\circ \psi$ is also a Fatou coordinate of $\tilde{P}_j$. The uniqueness of Fatou coordinates therefore implies that we can normalize so that
		$$\tilde{\phi}_j^{-1}= \psi_{j+s}\circ \tau\circ \Phi_{j+s}^{-1} = \psi_{j+s}^{-1}$$
		 on $\phi_{j+s}(P_{j+s})\cap \tilde \phi_j(\tilde P_j).$
	\end{proof}
	
	\begin{proof}[Proof of Proposition \ref{continuous dependence}]
		Part 4 of Proposition \ref{fatou coordinates} implies that each $\Phi_j$ depends continuously and holomorphically on $f$; the other dependences follow immediately.
	\end{proof}
	
	\begin{proof}[Proof of Proposition \ref{horn map parabolic}]
		It follows from the definition that the horn maps of $f$ are given by 
		$$H_+^f(w)= \Phi_{0}\circ G_1^{k(w)}\circ \Phi_{1}^{-1}(w)$$
		on $\QQ^1\cap \QQ^0$, where $k: \QQ^1\cap \QQ^0\to \{0, 1, \dots, q-1\}$ is continuous.
		The desired properties then follow  from the estimate of Fatou coordinates in part (5) of Proposition \ref{fatou coordinates}.		
	\end{proof}

	\subsection{Near-parabolic flowers}
	
	We keep  $f$ as above and fix some $t_0>0$. Let  $g$ be an analytic map defined on $V$ such that $g(0) = 0$, we will first restrict to the case where $g'(0)=\Exp\circ \mu_{p/q}^+ (\alpha)$ for some $\alpha\in A_{t_0}$. 
	On $V$, for each $n\geq 0$  we can write 
	$$g^n(z) = e^{2\pi i n(\mu_{p/q}^+(\alpha)-p/q)}f^n(z)(1+u_n(z))$$
	for some holomorphic function $u_n: V\to \C$ that converges locally uniformly to zero when $g\to f$. 
%
	In particular, $g^q$ has $q$ non-zero fixed points $(\sigma_{2j})_{j\in \Z/q\Z}$ that each satisfy
	\begin{equation}\label{sigma}
		q\sigma_j^q = (1-\Exp(q \mu_{p/q}^+(\alpha)))(1+o(1))
	\end{equation}
	when $g\to f$; we label these fixed points so that $\sigma_{2j+2}= g(\sigma_{2j})\approx e^{2\pi i /q}\sigma_{2j}.$
	We set 
	$$\sigma= q^2\prod_{j\in \Z/q\Z}\sigma_j \text{ and }\alpha' = \frac{q\alpha}{q+q_+'\alpha},$$
	so it follows from \eqref{sigma} and Proposition \ref{prop:mobius form} that
	$$\sigma= -2\pi i\alpha'(1+o(1))$$
	when $g\to f$.
	
	To make the $q\geq 1$ setting similar to the $q=1$ case, we first make the following observation:
	\begin{prop}
		There is a degree $\leq 2q-1$ polynomial $\psi_g:\C\to \C$ satisfying:
		\begin{enumerate}
			\item $\psi_g(z)/z^q = O(1)$ near $z=0$.
			\item $\psi_g(\sigma_{2j})=\sigma$ for all $j$. 
			\item $\psi_g\to \psi$ when $g\to f$.
		\end{enumerate}
	\end{prop}
	
	\begin{proof}
		Let $A = \{a_0, \dots, a_{q-1}\}$ be a set of $q$ distinct points in $\C^*$. First we observe that there is a polynomial $\psi_A$ such that $\psi_A(z)/z^q=O(1)$ near $z=0$, $\psi_A(a_j) =1$ for all $j$, and $\psi_A(z)\to z^q$ locally uniformly when $A$ converges to the $q$-th roots of unity. 
		Writing
		$$\psi_A(z) =  z^q(s_0+s_1z+\cdots s_{q-1}z^{q-1}),$$
		we can find $\psi_A$ by solving the system of $q$ linear equations determined by 
		$$\psi_A(a_0) = \psi_A(a_1)=\cdots \psi_A(a_{q-1})= 1$$
		for $s_0, \dots, s_{q-1}$. 
		As the points in $A$ are all distinct, this system of equations has a unique solution that  depends analytically on $A$. When $A$ is the set of $q$-th roots of unity, $\psi_A(z) = z^q$ solves the system, hence $\psi_A(z) \to z^q$ locally uniformly when $A$ converges to the $q$-th roots of unity.
		
		Now we set $A = \{1,\sigma_{2}/\sigma_0 \dots, \sigma_{2(q-1)}/\sigma_0\}$ and $$\psi_g(z)= \sigma \psi_A(z/\sigma_0).$$
		As $\sigma_{2j}/\sigma_0\to e^{2\pi ij/q}$ and $\sigma/ \sigma_0^q\to q^2$ when $g\to f$, the desired properties of $\psi_g$ follow from the properties of $\psi_A$ above.
	\end{proof}

	Using the map $\psi_g$, for all $n$ we define the multi-valued map
	$$\tilde{g}^n:= \psi_g\circ g^n\circ \psi_g^{-1}.$$
	Thus each branch of $\tilde{g}^n$ fixes zero and $\sigma$. 	
	As for $\tilde{f}^n$ in the previous subsection, the superscript for $\tilde{g}^n$ does not indicate an iterate as $\tilde{g}^2\neq \tilde{g}^1\circ \tilde{g}^1$.
	
	We have the universal covering map 
	$\tau_g: \C\to \chat\sm \{0, \sigma\}$
	given by 
	$$\tau_g(z) = \frac{\sigma}{1- \Exp(-\alpha'w)}.$$
	The group of deck-transformations of these maps is generated by the translation $T_g:=T_{1/\alpha'}$.
	Using the Taylor series expansion of $\Exp$, we observe that for any $w\in \C$ we have
	$$\tau_g(z)=\frac{\sigma}{1-\Exp(-\alpha'w)} = \frac{\sigma}{2\pi i \alpha' w}(1+O(\alpha'w))\to -\frac{1}{w}=\tau(z)$$
	when $g\to f$.
	
	We now recall the sets $\QQ^0 = \QQ(\xi, +\infty, t_0)$ and $\QQ^1= \QQ(-\infty, -\xi, t_0)$ of the previous subsection. We denote
	$$\QQ^0_g = \QQ(\xi, -\xi+ 1/\alpha', t_0).$$
	We therefore have that
	 $$\QQ_g^{0}\to \QQ^0 \text{ and }\QQ_g^1 = T^{- 1}_g(\QQ_g^{0})\to \QQ^1$$
	when $g\to f$. 
	
	\begin{figure}
		\begin{center}
			\def\svgwidth{5in}
\begingroup%
  \makeatletter%
  \providecommand\color[2][]{%
    \errmessage{(Inkscape) Color is used for the text in Inkscape, but the package 'color.sty' is not loaded}%
    \renewcommand\color[2][]{}%
  }%
  \providecommand\transparent[1]{%
    \errmessage{(Inkscape) Transparency is used (non-zero) for the text in Inkscape, but the package 'transparent.sty' is not loaded}%
    \renewcommand\transparent[1]{}%
  }%
  \providecommand\rotatebox[2]{#2}%
  \newcommand*\fsize{\dimexpr\f@size pt\relax}%
  \newcommand*\lineheight[1]{\fontsize{\fsize}{#1\fsize}\selectfont}%
  \ifx\svgwidth\undefined%
    \setlength{\unitlength}{1063.1899664bp}%
    \ifx\svgscale\undefined%
      \relax%
    \else%
      \setlength{\unitlength}{\unitlength * \real{\svgscale}}%
    \fi%
  \else%
    \setlength{\unitlength}{\svgwidth}%
  \fi%
  \global\let\svgwidth\undefined%
  \global\let\svgscale\undefined%
  \makeatother%
  \begin{picture}(1,0.81361981)%
    \lineheight{1}%
    \setlength\tabcolsep{0pt}%
    \put(0,0){\includegraphics[width=\unitlength]{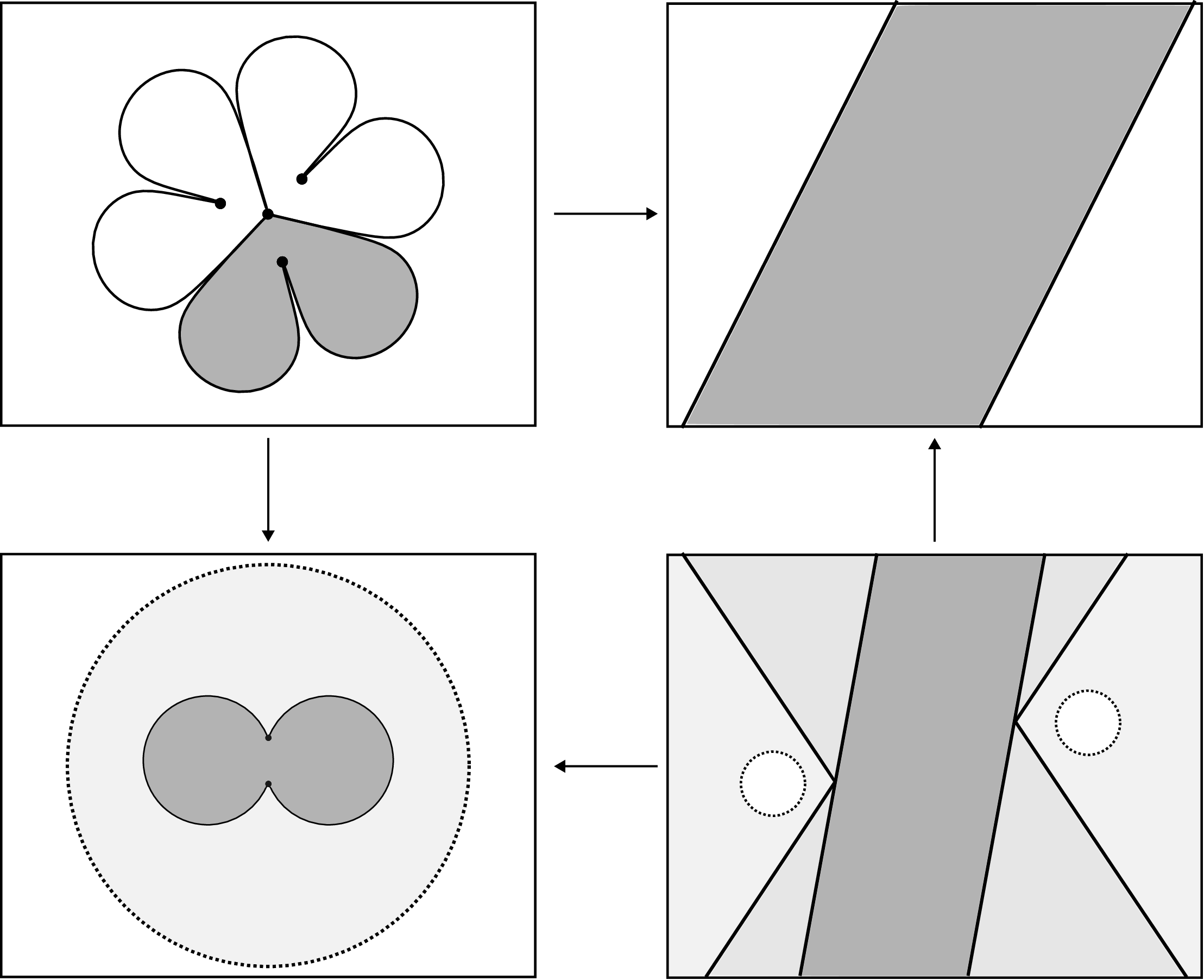}}%
    \put(0.47535833,0.65793358){\makebox(0,0)[lt]{\lineheight{1.25}\smash{\begin{tabular}[t]{l}$\phi_{g, 0}$\end{tabular}}}}%
    \put(0.48954653,0.19759227){\makebox(0,0)[lt]{\lineheight{1.25}\smash{\begin{tabular}[t]{l}$\tau_g$\end{tabular}}}}%
    \put(0.23982164,0.40140151){\makebox(0,0)[lt]{\lineheight{1.25}\smash{\begin{tabular}[t]{l}$\psi_g$\end{tabular}}}}%
    \put(0.78978187,0.40140151){\makebox(0,0)[lt]{\lineheight{1.25}\smash{\begin{tabular}[t]{l}$\Phi_{g, 0}$\end{tabular}}}}%
    \put(0.89289638,0.02408978){\makebox(0,0)[lt]{\lineheight{1.25}\smash{\begin{tabular}[t]{l}$\QQ_g^0$\end{tabular}}}}%
    \put(0.63750105,0.15603956){\makebox(0,0)[lt]{\lineheight{1.25}\smash{\begin{tabular}[t]{l}\tiny{$0$}\end{tabular}}}}%
    \put(0.88617107,0.20778729){\makebox(0,0)[lt]{\lineheight{1.25}\smash{\begin{tabular}[t]{l}\tiny{$1/\alpha$}\end{tabular}}}}%
    \put(0.32711228,0.49334609){\makebox(0,0)[lt]{\lineheight{1.25}\smash{\begin{tabular}[t]{l}$P_{g, 0}$\end{tabular}}}}%
  \end{picture}%
\endgroup%

			\caption{Construction of a near-parabolic flower. Note that the petals should be spiraling around the fixed points, we omit this detail.}
			\label{petal def fig}
		\end{center}
	\end{figure}

	For any even integer $j$, there is a unique branch of $\psi_g^{-1}\circ \tau_g$
	that sends  $w\in \QQ_g^{0}$ with $\text{Im}\,\alpha'w\ll 0$ close to $\sigma_j$; we will denote this branch by $\psi_{g, j}^{-1}\circ \tau_g$ and analytically extend it to all of $\QQ_g^{0}$. 
	As $\tau_g\to \tau$ and $\psi_g\to \psi$, we can choose our labeling of the cycle $(\sigma_j)$ so that $\psi_{g, 0}^{-1}\circ \tau_g\to \psi_0^{-1}\circ \tau$ on $\QQ^0$ when $g\to f$.
	As $\sigma_{j+2}\approx e^{2\pi i /q}\sigma_j$ and $\psi_{j+2}^{-1} = e^{2\pi i /q}\cdot \psi_j^{-1}$, it follows that $\psi_{g, j}^{-1}\circ \tau_g\to \psi_j^{-1}\circ \tau$ on $\QQ^0$ when $g\to f$ for all even $j$. 	
	Moreover, these branches are also related in the following way:
	
	\begin{prop}\label{same branch}
		For any even $j$,  
		$$\psi_{g, j}^{-1}\circ \tau_g\circ T_g=\begin{cases}
			\psi_{g, j+2}^{-1}\circ \tau_g&\text{ on the upper component of }\QQ_g^0\cap T_g^{-1}(\QQ_g^0),\\
			\psi_{g, j}^{-1}\circ \tau_g&\text{ on the lower component of }\QQ_g^0\cap T_g^{-1}(\QQ_g^0).
		\end{cases}$$
	\end{prop}

	\begin{proof}
		Let us fix some $w$ in the upper component of $\QQ_g^0\cap T_g^{-1}(\QQ_g^0)$. Let $\gamma_{0}$ be a path in $\QQ_g^0$ connecting $-i\infty$ to $T_g(w)$, and let $\gamma_2$ be a path in $T_g(\QQ_g^0)$ connecting $T_g(w)$ to $-i\infty$. Let $\gamma$ be the concatenation of $\gamma_0$ and $\gamma_2$. 
		
		As $\tau_g$ sends every integer multiple of $1/\alpha'$ to $\infty$ and $\gamma$ winds once clockwise around $1/\alpha'$, the path $\tau_g\circ \gamma$ connects $\sigma$ to itself and winds once counterclockwise around zero. 
		Thus we can analytically extend $\psi_{g, j}^{-1}\circ \tau_g$ on $\gamma_0$ to all of $\gamma$; the resulting curve $\psi_{g, j}^{-1}\circ \tau_g\circ \gamma$ is the unique lift of $\gamma$ by $\psi_g^{-1}\circ \tau_g$ that connects $\sigma_j$ to $\sigma_{j+2}$. We can similarly extend 
		$\psi_{g, j+2}^{-1}\circ \tau_g\circ T_g^{-1}$ 
		on $\gamma_2$ to all of $\gamma$; the resulting curve  $\psi_{g, j+2}^{-1}\circ \tau_g\circ T_g^{-1}\circ \gamma$ 
		is also the unique lift of $\gamma$ by $\psi_g^{-1}\circ \tau_g$ that connects $\sigma_j$ to $\sigma_{j+2}$. 
		Hence $\psi_{g, j}^{-1}\circ \tau_g\circ T_g(w) = \psi_{g, j+2}^{-1}\circ \tau_g(w)$. 
		
		For $w$ in the lower component of $\QQ_g^0\cap T_g^{-1}(\QQ_g^0)$,
		set $\tilde{z} = \tau_g(w)$.
		For 
		 $\text{Im}\,w\ll 0$, 
		 by definition both $\psi_{g, j}^{-1}\circ \tau_g(w)$ and $\psi_{g, j}^{-1}\circ \tau_g\circ T_g^{-1}(w)$  are the unique branch of 
		 $\psi_g^{-1}(\tilde{z})$ close to $\sigma_j$, hence $\psi_{g, j}^{-1}\circ \tau_g\circ T_g(w)=\psi_{g, j}^{-1}\circ \tau_g(w)$. By analytic continuation the equation holds all the entire lower component of $\QQ_g^0\cap T_g^{-1}(\QQ_g^0)$.
	\end{proof}
	
	For any odd integer $j$, we set $$\psi_{g, j}^{-1}\circ \tau_g= \psi_{g, j-1}^{-1}\circ \tau_g\circ T_g$$ on $\QQ_g^1$. 
	So Proposition \ref{same branch} implies that $\psi_{g, j}^{-1}\circ \tau_g= \psi_{g, j+1}^{-1}\circ \tau_g$ on the upper component of $\QQ_g^0\cap \QQ_g^1$; it then follows from the definition of $\psi_{j}^{-1}$ that 
	$\psi_{g, j}^{-1}\circ \tau_g\to \psi_j^{-1}\circ \tau$ on $\QQ^1$ when $g\to f$.

	For all integers $j$ and $w\in \QQ_g^{[j]}$ and $z = \psi_{g, j}^{-1}\circ \tau_g(w)$, we define
	$${G}^n_j(w) = w+\frac{1}{2\pi i \alpha'}\log\left(\frac{\psi_g\circ g^n(z)}{\psi_g\circ g^n(z)- \sigma}\cdot \frac{\psi_g(z)-\sigma}{\psi_g(z)}\right),$$
	 using the branch of log with imaginary part in $[-\pi, \pi).$
	 We have the following properties of $G_j^n$:

	\begin{prop}\label{prop:lifted perturbed}
		For all $1\leq n \leq q$ and all $j$,  if $g$ is close enough to $f$ then we have:
		\begin{enumerate}
			\item $\psi_{g,j+2np}^{-1}\circ \tau_g\circ {G}^n_j =g^n\circ \psi_{g,j}^{-1}\circ \tau_g$ and $$G_j^n\circ T_g = \begin{cases}
				 T_g\circ G^{n}_{j+2} &\text{ on the upper component of }\QQ_g^{[j]}\cap T_g^{-1}(\QQ_g^{[j]}),\\
				 T_g\circ G^{n}_{j} &\text{ on the lower component of }\QQ_g^{[j]}\cap T_g^{-1}(\QQ_g^{[j]}).
			\end{cases}$$
			\item For $w\in \QQ_g^{[j]}$, 
			$$|{G}^n_j(w) - (w+n/q)|< \epsilon \text{ and }|( {G}_j^n)'(w) -1|< \epsilon.$$
			\item ${G}^q_j(w) = w+1+O(1/w^2)$ when $\emph{Im}\,w\to \infty$
			\item ${G}_j^n$ converges to ${F}^n_j$ on $\QQ^{[j]}$ when $g\to f$.
		\end{enumerate}
	\end{prop}

	\begin{proof}
		Using the Taylor series expansion of $\log(1+x)$, we can estimate 
		\begin{align*}
			{G}_j^n(w) &= w+\frac{1}{2\pi i \alpha'}\left(\frac{\psi_g\circ g^n(z)}{\psi_g\circ g^n(z)- \sigma}\cdot \frac{\psi_g(z)-\sigma}{\psi_g(z)}-1\right)+o(1)\\
			&=w+\frac{-\sigma(\psi_g\circ g^n(z) - \psi_g(z))}{2\pi i \alpha'(\psi_g\circ g^n(z) - \sigma)\psi_g(z)}+o(1)\\
			&\to -\frac{1}{\tau(w)}+\frac{\tilde{f}^n\circ \tau(w) -\tau(w)}{\tau(w)\cdot \tilde{f}^n\circ \tau(w)}\\
			& = \tau\circ \tilde{f}^n\circ \tau(w) = {F}^n(w)
		\end{align*}
		when $g\to f$ for some branch of $F^n$;  part (4) of the proposition with then follow from part (1) of the proposition and the convergence of $g\to f$ and $\psi_{g, j}^{-1}\circ \tau_g\to \psi_j^{-1}\circ\tau$. As every branch of ${F}^n$ satisfies $$|{F}^n(w) - (z-n/q)|< \epsilon/2 \text{ and }|({F}^n)'(z) - 1|< \epsilon/2$$ on $\QQ^0$, part $(2)$ of the proposition follows.
		
		For $z = \psi_{g, j}^{-1}\circ \tau_g(w)$, so $\psi(z) = \tau_g(w)$, using the definition of $\tau_g$ we can compute explicitly 
		\begin{align*}
			\tau_{g}\circ {G}_j^n(w) &= \frac{\sigma\cdot \psi_g\circ g^n(z)(\psi_g(z)-\sigma)}{\sigma\cdot \psi_g\circ g^n(z)(\psi_g(z)-\sigma)-(\psi(z)-\sigma)(\psi_g\circ g^n(z)-\sigma)}\\
			&=\psi_g\circ g^n(z).
		\end{align*}
		As $G_j^n(w) \approx w+n/q$ and $$g^n(z) \approx e^{2\pi i np/q}z \approx \psi_{g,j+2np}^{-1}(w),$$ we can conclude that  $\psi_{g, j+2np}^{-1}\circ \tau_g\circ G^n_j = g^n\circ \psi_{g, j}^{-1}\circ \tau_g$; this proves the first equation in part $(1)$ of the proposition. The second equation in part (1)
		follows immediately from the definition of $G_j^n$ and Proposition \ref{same branch}.

		For part $(3)$ of the proposition, we observe that $\tau_j(w) \to 0$ when $\text{Im}\, w\to \pm \infty$ and $(G^q)'(0) = \Exp(q\alpha').$ Hence
		$${G}_j^q(w)-(w+1)\to 0$$
		when $\text{Im}\, w\to \pm \infty$. As ${G}_j^q$ is periodic under $T_g^q$ by the above, we can compute the Fourier expansion to see that the decay can be at most $O(1/w^2)$.
	\end{proof}
	
 	While the superscript of $G^n$ does not indicate an iterate, part (1) of Proposition \ref{prop:lifted perturbed} implies that for any $j$ and $n$, 
 	$$G_{j+2np}\circ G_j^n = G_{j+2p}^n\circ G_j$$
 	wherever both sides of the equation are defined.

 	For all integers $j$,  it follows from Propositions \ref{fatou coordinates} and \ref{prop:lifted perturbed} that there is an analytic map $\Phi_{g, j}: \QQ_g^{[j]}\to \C$ conjugating ${G}_j^q$ to $T_1$.
 	Let $S_{g, j}$ be a strip in $\Phi_{g, j}(\QQ_g^{[j]})$ for all $j$, and set 
 	$$P_{g, j} = \psi_{g, j}^{-1}\circ \tau_g\circ \Phi_{g, j}^{-1}(S_{g, j})\text{ and }\phi_{g, j} = \Phi_{g, j}\circ \tau_g^{-1}\circ \psi_g.$$
 	The following proposition immediately proves Theorem \ref{thm:near-parabolic flowers}:

 	\begin{prop}
 		We can choose the strips $S_{g, j}$ such that $(P_{g, j}, \Phi_{g, j})$ is a near-parabolic flower for $g$ relative to $f$.
 	\end{prop}
 	
 	\begin{proof}
%
%
	Fixing some $z_j\in P_j$ and setting $w_{j} = \tau^{-1}\circ \psi(z_j)$, when $g$ is close to $f$ there is a unique $w_j\in \tau_g^{-1}\circ \psi_g(z_j)$ close to $w_j$. Normalizing the Fatou coordinates so that $\Phi_{g, j}(w_{g, j}) = \Phi_{j}(w_j)$, it follows from Proposition \ref{fatou coordinates} that $\Phi_{g, j}\to \Phi_{j}$ when $g\to f$.
	
	When $q = 1$, $t_0 = 1$, and $\alpha\in \R$, it is shown in \cite[Lemma 17]{positive_area} that for any $R>0$  
	the image of $\Phi_{g, 0}$ contains the vertical strip 
	$$\QQ(\Phi_{g, 0}(\xi+R), \Phi_{g, 0}(-\xi-R+1/\alpha'), 0)$$
	and $\tau_g\circ \Phi_{g, 0}^{-1}$ is injective on this strip
	when $\epsilon$ is sufficiently small (which corresponds to taking $\xi$ large and $g$ close to $f$). 
	Fixing some $0< t_0'< t_0$, the same argument implies in our setting that the image of 
	$\Phi_{g, 0}$ contains the the region
	$$\QQ(\Phi_{g, 0}(\xi+R), \Phi_{g, 0}(-\xi-R+1/\alpha'), t_0')$$
	and $\tau_g\circ \Phi_{g, 0}^{-1}$ is injective on the maximal strip with tilt $t\in (-t_0', t_0)$ in this region when $\epsilon$ is sufficiently small; we define $S_{g, 0}$ to be this strip. 
	Hence $P_{g, 0}$ is a Jordan domain with both zero and $\sigma_{g, 0}$ on its boundary, and $\phi_{g, 0}$ is a Fatou coordinate on $P_{g, 0}$. Choosing the other petals so that 
	$P_{g, j+2}= g(P_{g, j})$ for $j\neq 0$, the only hypothesis to check that $(P_{g, j}, \phi_{g, j})$ is a near-parabolic flower for $g$ relative to $f$ is that $\overline{P_{g, j}}\cap \overline{P_{g, j'}}=\{0\}$ for any even $j, j'$.  
	
	We consider  $P_{g, j}$ and $g^m(P_{g, j})$ for some even $j$ and  some $0< m< q$. 
	Setting $\tilde S = \Phi_{g, j}^{-1}(S_{g,j})$ and $\tilde{S}' = T_g^{mp}\circ G^m_j(\tilde{S})$,
	it follows from Propositions \ref{same branch} and \ref{prop:lifted perturbed} that we can extend $\psi_{g, j}^{-1}\circ \tau_g$ so that it maps $\tilde{S}$ and $\tilde{S}'$ to 
	$P_{g, j}$ and $g^m(P_{g, j})$ respectively. 
	Note that if the closures of $P_{g, j}$ and $g^m(P_{g, j})$ have some non-zero point in their intersections, then this point must be close to zero. Indeed this follows from the convergence of $\phi_{g, j}\to \phi_j$ and $\phi_{g, j+1}\to \phi_{j+1}$ when $g\to \infty$. Hence if the closures of $P_{g, j}$ and $g^m(P_{g, j})$ have some non-zero point in their intersections, then the intersection of the closures of $\tilde{S}$ and $\tilde{S}''= T_g^{kq}(\tilde{S}')$ must be non-empty for some integer $k$. We can analytically extend $\Phi_{g, j}$ to both $\tilde{S}$ and $\tilde{S}''$; the uniqueness of Fatou coordinates implies that the image of both sets are strips with tilt $t$. The asymptotic estimate of $\Phi_{g, j}$ near $\infty$ and the definition of $S_{g, j}$ together imply that $\tilde{S}$ and $\tilde{S}''$ have disjoint closures.
 	\end{proof}

 	\begin{proof}[Proof of Theorem \ref{thm:near-parabolic flowers}]
 		The collection $(P_{g, j}, \Phi_{g, j})$ is a desired flower.
 	\end{proof}

 	We can also now prove some propositions from Section \ref{sec:near-parabolic}.

 	\begin{proof}[Proof of Proposition \ref{prop:petals perturbed}]
 		The argument is similar to  the proof of Proposition \ref{prop:petals}: for another flower $(\tilde P_{g, j}, \tilde \phi_{g, j})$ we can lift the petal $\tilde P_{g, j}$ by some  $\psi_{g, j+s}^{-1}\circ \tau_g$ and extend $\Phi_{g, j+s}$ to the lift. 
 		The uniqueness of the Fatou coordinates allows us to normalize so that 
 		$$\tilde \phi_{g, j}^{-1} = \psi_{g, j+s}^{-1}\circ \tau_g\circ \Phi_{g, j+s}^{-1} = \phi_{g, j+s}^{-1}$$
 		on $\tilde \phi_{g, j}(\tilde P_{g, j})\cap \phi_{g, j+s}(P_{g, j+s})$. When $g\to f$, the convergence of the Fatou coordinates to Fatou coordinates for $f$ ensures that this intersection is non-empty. 
 	\end{proof}
 	\begin{proof}[Proof of Proposition \ref{prop:near-parabolic independence}]
 		Note that our construction of the flower above does not strongly depend on $f$: we only use $f$ to get uniform control of $\epsilon$. In particular, if $\tilde{f}$ is another map with a $p/q$-parabolic fixed point at zero and $\tilde{f}$ is sufficiently close to $f$, then the construction of a flower for $g$ relative to $\tilde{f}$ is identical.
 	\end{proof}
 	\begin{proof}[Proof of Proposition \ref{prop:tilt}]
 		 For $q = 1$, $t=0$, $t_0 = 1$, and $\alpha\in \R$ this is proved in \cite[Lemma 17]{positive_area};
 		our argument will be the same.
 		We set
 		$$\Gamma_g^- = \left\{\frac{\log \alpha}{2\pi i}-is: s>0\right\}\text{ and }\Gamma_g^+ = \left\{\frac{\log \alpha}{2\pi i}+\alpha'(i+t)s: s>0\right\},$$
 		using the branch of $\log\alpha$ with imaginary part in $[-\pi, \pi)$.
 		For $\Gamma_g = \Gamma_g^-\cup \Gamma_g^+$, first we show:
 		\begin{lemma}\label{lem:first log}
 			There is a continuous branch of $\log\circ\tau_g$ defined on $\Phi^{-1}(S_{g, 0})$ such that
 			$$\sup_{w\in \Phi^{-1}(S_{g, 0})}\emph{dist}\left(\frac{\log\circ \tau_g(w)}{2\pi i }, \Gamma_g\right)< {M}$$
 			for some uniform $M>0$ when $g\to f$.
 		\end{lemma}
 		
 		\begin{proof}
 			We  set $$\Omega_g= \{w\in \C:\text{Im}\, \alpha'w>0\}.$$
 			As $\tau_g$ sends the interval $(0, 1/\alpha')$ to the perpendicular bisector of the interval $(0, \sigma)$, and as $$\sigma =-2\pi i\alpha(1+o(1))$$ when $g\to f$, there is a  continuous branch of ${\log\circ \tau_g}$ defined on $\C\sm \Omega_g$ satisfying
 			$$\sup_{w\in \C\sm \Omega_g}\text{dist}\left(\frac{\log\circ \tau_g(w)}{2\pi i}, \Gamma_g^-\right)< M^-$$
 			for some constant $M^-$ that does not depend on $g$.
 			Let $\ell_{g, 0}$ denote the intersection of the line through $1/2\alpha'$ with slope $1/t$ with $\Omega_g$, and set $\ell_{g, k}= T_g^k(\ell_{g, 0})$ for all $k$. 
 			As 
 			$$\frac{\log\circ \tau_g(w)}{2\pi i }=  \alpha'w+\frac{\log \sigma}{2\pi i}+\frac{\log(\Exp(\alpha'w)-1)}{2\pi i},$$
 			there is a continuous branch of $\log\circ \tau_g$ defined on $\Omega_g$ such that
 			$$\sup_k\sup_{w\in \ell_{g, k}}\left(\frac{\log\circ \tau_g(w)}{2\pi i}, T_k(\Gamma_g^+)\right)< C^+$$
 			for some constant $C^+$ that depends only on $t$ and $t_0$. So to complete the proof, it suffices to show that $\Omega_g\cap\Phi_{g, 0}^{-1}(S)$ intersects only bounded number of the lines $\ell_{g, k}$, or equivalently
 			$$\sup_{w\in \Omega_g\cap \Phi_{g, 0}^{-1}(S_{g,0})}\text{dist}(w, \ell_{g,0})< \frac{M^+}{|\alpha'|},$$
 			when $g$ is close to $f$ for some constant $M^+$ that does not depend on $g$.
 			This follows immediately from the fact that 
 			$$\sup_{w\in \Omega_g\cap \QQ_g^0}|\Phi_{g, 0}(w)-w|= O(1/\alpha')$$
 			when $g\to f$. This fact is proved in \cite[Lemma 17]{positive_area} when $q = 1$, $t=0$, $t_0 = 1$, and $\alpha\in \R$;  in our general setting the argument is identical. 
 		\end{proof}
 		
 		As $\psi_g$ maps $P_{g, 0}$ univalently into $\tau_g\circ \Phi_{g, 0}^{-1}({S}_{g, 0})$ and converges to $\psi(z) =q^2z^q$, it follows from Lemma \ref{lem:first log} that there is a branch of $\log$ defined on $P_{g, 0}$ such that
 		$$\sup_{z\in P}\text{dist}\left(\frac{\log z}{2\pi i}, \Gamma_\alpha\right)< M$$
 		when $g\to f$ 
 		for some constant $M$ that does not depend on $g$, where 
 		$$\Gamma_\alpha= \{w: qw \in \Gamma_g\}.$$
 		Note that $\Gamma_g^-$ is contained in a strip of tilt $0$, and $\Gamma_g^+$ is contained in a strip with tilt $t'$, where 
 		$$t' = \frac{t\text{Re}\,\alpha'-\text{Im}\,\alpha'}{t\text{Im}\, \alpha'+\text{Re}\,\alpha'}$$
 		As every point in $\tilde P=( \log P_{g, 0})/2\pi i$ is distance at most one from $\Gamma_\alpha$, and $\tilde{P}$ is also contained in an upper half-plane that depends only on $V$, it follows that $\tilde{P}$ is contained in a strip with tilt $t'$ and width $t'|\log|\alpha||+2M$
 		as desired.
 	\end{proof}
 	
 	\begin{proof}[Proof of Proposition \ref{prop:phase}]
 		Writing $\phi_{g, 1}= T_\lambda\circ \phi_{g, 0}$ for some $\lambda\in \C$, it follows from the definition that $$\Phi_{g, 1} = T_\lambda\circ \Phi_{g, 0}\circ T_g.$$
 		Let $H_g$ be the corresponding horn map for $g$ relative to $f$.
 		Fixing some $x\in \R$ and $w = x+it$ for $t>0$, when $g$ is close to $f$ and $t\gg 0$ we have
 		$$H_g(w) = \Phi_{g, 0}\circ G_1^{q_+'}\circ \Phi_{g, 1}^{-1}(w).$$
 		Normalizing the Fatou coordinates so that $H_g(w)-w\to 0$ when $t\to \infty$, as $\Phi_{g, 0}$ and $G_1^{q_+'}$ tend to a translation and  to $T_{q_+'/q}$ respectively near $\infty$, it follows that 
 		$$0=\frac{q_+'}{q}-\frac{1}{\alpha'}-\lambda= -\frac{1}{\alpha}-\lambda,$$
 		hence $\lambda = -1/\alpha$.
 		
 		With the above normalization, we need to check that we still have $\phi_{g, 1}\to \phi_1$ when $g\to 1$. When $q = 1$ this is proved in  in \cite[Lemma 3.4.2]{shishikura_1}; the same argument applies here.
 	\end{proof}

 	\begin{proof}[Proof of Proposition \ref{horn maps perturbed}]
 		The argument is almost identical to the proof of Proposition \ref{horn map parabolic}; the only difference is checking that the horn map $H_+^{g, f}$ is well-defined. The well-definedness follows similarly to the proof of Proposition \ref{prop:extensions}: if $\rho^g\circ \chi_+^{g}g(w-m_i)$ is defined for some distinct integers $m_1< m_2$, our definition of the extensions $\rho^g$ and $\chi_+^g$ ensures that
 		\begin{align*}
 			T_{-m_2}\circ \rho^g\circ \chi_+^g(w+m_2)&= T_{-m_2}\circ \rho^g\circ g^{(m_2-m_1)q}\circ \chi_+^g(w+m_1)\\
 			&= T_{-m_1}\circ \rho^g\circ g^{(m_2-m_1)q}\circ \chi_+^g(w+m_1),
 		\end{align*}
 		so $H_+^g(w)$ does not depend on the choice of $m_i$.
 	\end{proof}
 	
 	Let us now consider instead the case where $g'(0) = \Exp\circ \mu_{p/q}^-(\alpha)$. 
 	Setting $f^*(z) = \overline{f(\overline{z})}$ and $g^*(z) = \overline{g(\overline{z})}$, we observe that $f^*$ has a non-degenerate $-p/q$-parabolic fixed point at zero and $$(g^*)'(0)= \Exp\circ \mu_{-p/q}^+(\overline{\alpha}).$$ For $(P^*_{g^*, j}, \phi^*_{g, j})_{j\in \Z/q\Z}$ a near-parabolic flower for $g^*$ relative to $f^*$, note that collection $(P_{g, j}, \phi_{g, j})$ with $P_{g, j} = \{z: \overline{z}\in P_{g^*, -j}^*\}$  and $\phi_{g, j}(z) = \overline{\phi_{g^*, -j}^*(\overline{z})}$ is a near-parabolic flower for $g$ relative to $f$. With this description, the desired properties for the flower all follow easily.

	\renewcommand*{\bibfont}{\small}
	\printbibliography

@article {parabolic_PLY,
	AUTHOR = {Buff, Xavier and Epstein, Adam L.},
	TITLE = {A parabolic {P}ommerenke-{L}evin-{Y}occoz inequality},
	JOURNAL = {Fund. Math.},
	FJOURNAL = {Fundamenta Mathematicae},
	VOLUME = {172},
	YEAR = {2002},
	NUMBER = {3},
	PAGES = {249--289},
	ISSN = {0016-2736,1730-6329},
	MRCLASS = {37F10},
	MRNUMBER = {1898687},
	MRREVIEWER = {Peter\ Ha\"issinsky},
	DOI = {10.4064/fm172-3-3},
	URL = {https://doi.org/10.4064/fm172-3-3},
}

@ARTICLE{CS_satellite,
	author = {Cheraghi, Davoud and Shishikura, Mitshurio},
	title = "{Satellite renormalization of quadratic polynomials}",
	journal = {arXiv e-prints},
	year = 2015,
	month = sep,
	eid = {arXiv.1509.07843},
	pages = {arXiv.1509.07843},
	archivePrefix = {arXiv},
	eprint = {1509.07843},
	primaryClass = {math.DS},
	adsurl = {https://arxiv.org/abs/1509.07843v1},
}

@article {positive_area,
	AUTHOR = {Buff, Xavier and Ch\'{e}ritat, Arnaud},
	TITLE = {Quadratic {J}ulia sets with positive area},
	JOURNAL = {Ann. of Math. (2)},
	FJOURNAL = {Annals of Mathematics. Second Series},
	VOLUME = {176},
	YEAR = {2012},
	NUMBER = {2},
	PAGES = {673--746},
	ISSN = {0003-486X},
	MRCLASS = {37F50},
	MRNUMBER = {2950763},
	MRREVIEWER = {Peter Ha\"{\i}ssinsky},
	DOI = {10.4007/annals.2012.176.2.1},
	URL = {https://doi.org/10.4007/annals.2012.176.2.1},
}

@incollection {shishikura_1,
	AUTHOR = {Shishikura, Mitsuhiro},
	TITLE = {Bifurcation of parabolic fixed points},
	BOOKTITLE = {The {M}andelbrot set, theme and variations},
	SERIES = {London Math. Soc. Lecture Note Ser.},
	VOLUME = {274},
	PAGES = {325--363},
	PUBLISHER = {Cambridge Univ. Press, Cambridge},
	YEAR = {2000},
	MRCLASS = {37F45 (30D05)},
	MRNUMBER = {1765097},
	MRREVIEWER = {Carsten Lunde Petersen},
}

@article {shishikura_boundary,
	AUTHOR = {Shishikura, Mitsuhiro},
	TITLE = {The {H}ausdorff dimension of the boundary of the {M}andelbrot
	set and {J}ulia sets},
	JOURNAL = {Ann. of Math. (2)},
	FJOURNAL = {Annals of Mathematics. Second Series},
	VOLUME = {147},
	YEAR = {1998},
	NUMBER = {2},
	PAGES = {225--267},
	ISSN = {0003-486X},
	MRCLASS = {37F35 (30D05 37F45 37F50)},
	MRNUMBER = {1626737},
	MRREVIEWER = {Hartje Kriete},
	DOI = {10.2307/121009},
}

@book {Orsay2,
	AUTHOR = {Douady, A. and Hubbard, J. H.},
	TITLE = {\'{E}tude dynamique des polyn\^{o}mes complexes. {P}artie {II}},
	SERIES = {Publications Math\'{e}matiques d'Orsay [Mathematical Publications
	of Orsay]},
	VOLUME = {85},
	NOTE = {With the collaboration of P. Lavaurs, Tan Lei and P. Sentenac},
	PUBLISHER = {Universit\'{e} de Paris-Sud, D\'{e}partement de Math\'{e}matiques, Orsay},
	YEAR = {1985},
	PAGES = {v+154},
	MRCLASS = {58F08 (30D05 39B10)},
	MRNUMBER = {812271},
	MRREVIEWER = {M. Rees},
}

@book {Orsay1,
	AUTHOR = {Douady, A. and Hubbard, J. H.},
	TITLE = {\'{E}tude dynamique des polyn\^{o}mes complexes. {P}artie {I}},
	SERIES = {Publications Math\'{e}matiques d'Orsay [Mathematical Publications
	of Orsay]},
	VOLUME = {84},
	PUBLISHER = {Universit\'{e} de Paris-Sud, D\'{e}partement de Math\'{e}matiques, Orsay},
	YEAR = {1984},
	PAGES = {75},
	MRCLASS = {58F08 (30D05 39B10)},
	MRNUMBER = {762431},
	MRREVIEWER = {M. Rees},
}

@phdthesis{Lavaurs,
	author       = {Lavaurs, Pierre}, 
	title        = {Systemes dynamiques holomorphes: explosion de points periodiques
	paraboliques},
	school       = {These de doctrat de l'Universite de Paris-Sud},
	year         = 1989,
	address      = {Orsay, France}
}

@incollection {douady,
	AUTHOR = {Douady, Adrien},
	TITLE = {Does a {J}ulia set depend continuously on the polynomial?},
	BOOKTITLE = {Complex dynamical systems ({C}incinnati, {OH}, 1994)},
	SERIES = {Proc. Sympos. Appl. Math.},
	VOLUME = {49},
	PAGES = {91--138},
	PUBLISHER = {Amer. Math. Soc., Providence, RI},
	YEAR = {1994},
	MRCLASS = {58F23 (30D05)},
	MRNUMBER = {1315535},
	DOI = {10.1090/psapm/049/1315535},
}

@article{shishikura,
	title = {The renormalization for parabolic fixed points and their
	perturbation},
	author = {Inou, H. and Shishikura, M.},
	Year = {2008},
	Journal = {Manuscript},
}

@book {Yampolsky,
	AUTHOR = {Lanford, III, Oscar E. and Yampolsky, Michael},
	TITLE = {Fixed point of the parabolic renormalization operator},
	SERIES = {SpringerBriefs in Mathematics},
	PUBLISHER = {Springer, Cham},
	YEAR = {2014},
	PAGES = {viii+111},
	ISBN = {978-3-319-11706-5; 978-3-319-11707-2},
	MRCLASS = {37F25 (30D05 37F10)},
	MRNUMBER = {3308117},
	MRREVIEWER = {Artem Dudko},
	DOI = {10.1007/978-3-319-11707-2},
}

@book {Milnor,
	AUTHOR = {Milnor, John},
	TITLE = {Dynamics in one complex variable},
	SERIES = {Annals of Mathematics Studies},
	VOLUME = {160},
	EDITION = {Third},
	PUBLISHER = {Princeton University Press, Princeton, NJ},
	YEAR = {2006},
	PAGES = {viii+304},
	ISBN = {978-0-691-12488-9; 0-691-12488-4},
	MRCLASS = {37Fxx (30-01 30D05 37-01)},
	MRNUMBER = {2193309},
}

@incollection {Oudkerk,
	AUTHOR = {Oudkerk, Richard},
	TITLE = {The parabolic implosion: {L}avaurs maps and strong convergence
	for rational maps},
	BOOKTITLE = {Value distribution theory and complex dynamics ({H}ong {K}ong,
	2000)},
	SERIES = {Contemp. Math.},
	VOLUME = {303},
	PAGES = {79--105},
	PUBLISHER = {Amer. Math. Soc., Providence, RI},
	YEAR = {2002},
	MRCLASS = {37F10 (30D05 37F35 37F50 39B12)},
	MRNUMBER = {1943528},
	MRREVIEWER = {Peter Ha\"{\i}ssinsky},
	DOI = {10.1090/conm/303/05239},
}

@article {MMYconjecture,
	AUTHOR = {Cheraghi, Davoud and Ch\'{e}ritat, Arnaud},
	TITLE = {A proof of the {M}armi-{M}oussa-{Y}occoz conjecture for
	rotation numbers of high type},
	JOURNAL = {Invent. Math.},
	FJOURNAL = {Inventiones Mathematicae},
	VOLUME = {202},
	YEAR = {2015},
	NUMBER = {2},
	PAGES = {677--742},
	ISSN = {0020-9910},
	MRCLASS = {37F50 (11A55 30B70 30D05)},
	MRNUMBER = {3418243},
	MRREVIEWER = {Haifeng Chu},
	DOI = {10.1007/s00222-014-0576-2},
}

@misc{Shishkura-Yang,
	doi = {10.48550/ARXIV.1608.04106},
	
	url = {https://arxiv.org/abs/1608.04106},
	
	author = {Shishikura, Mitsuhiro and Yang, Fei},
	
	title = {The high type quadratic Siegel disks are Jordan domains},
	
	publisher = {arXiv},
	
	year = {2016},
	
	copyright = {arXiv.org perpetual, non-exclusive license}
}

@article {Cheraghi13,
	AUTHOR = {Cheraghi, Davoud},
	TITLE = {Typical orbits of quadratic polynomials with a neutral fixed
	point: {B}rjuno type},
	JOURNAL = {Comm. Math. Phys.},
	FJOURNAL = {Communications in Mathematical Physics},
	VOLUME = {322},
	YEAR = {2013},
	NUMBER = {3},
	PAGES = {999--1035},
	ISSN = {0010-3616},
	MRCLASS = {37F50},
	MRNUMBER = {3079339},
	MRREVIEWER = {Claire Chavaudret},
	DOI = {10.1007/s00220-013-1747-5},
}

@article {Cheraghi19,
	AUTHOR = {Cheraghi, Davoud},
	TITLE = {Typical orbits of quadratic polynomials with a neutral fixed
	point: non-{B}rjuno type},
	JOURNAL = {Ann. Sci. \'{E}c. Norm. Sup\'{e}r. (4)},
	FJOURNAL = {Annales Scientifiques de l'\'{E}cole Normale Sup\'{e}rieure. Quatri\`eme
	S\'{e}rie},
	VOLUME = {52},
	YEAR = {2019},
	NUMBER = {1},
	PAGES = {59--138},
	ISSN = {0012-9593},
	MRCLASS = {37F45 (37F10 37F25 37F50)},
	MRNUMBER = {3940907},
	MRREVIEWER = {Haifeng Chu},
	DOI = {10.24033/asens.2384},
}

@misc{Cheraghi17,
	doi = {10.48550/ARXIV.1706.02678},
	
	url = {https://arxiv.org/abs/1706.02678},
	
	author = {Cheraghi, Davoud},
	
	title = {Topology of irrationally indifferent attractors},
	
	publisher = {arXiv},
	
	year = {2017},
	
	copyright = {arXiv.org perpetual, non-exclusive license}
}

@article {AC18,
	AUTHOR = {Avila, Artur and Cheraghi, Davoud},
	TITLE = {Statistical properties of quadratic polynomials with a neutral
	fixed point},
	JOURNAL = {J. Eur. Math. Soc. (JEMS)},
	FJOURNAL = {Journal of the European Mathematical Society (JEMS)},
	VOLUME = {20},
	YEAR = {2018},
	NUMBER = {8},
	PAGES = {2005--2062},
	ISSN = {1435-9855},
	MRCLASS = {37F50 (28D05 30D05 37A50 37F25)},
	MRNUMBER = {3854897},
	MRREVIEWER = {Y\^{u}suke Okuyama},
	DOI = {10.4171/JEMS/805},
}

@article {Cheritat22,
	AUTHOR = {Ch{\'{e}}ritat, Arnaud},
	TITLE = {Near parabolic renormalization for unicritical holomorphic
	maps},
	JOURNAL = {Arnold Math. J.},
	FJOURNAL = {Arnold Mathematical Journal},
	VOLUME = {8},
	YEAR = {2022},
	NUMBER = {2},
	PAGES = {169--270},
	ISSN = {2199-6792},
	MRCLASS = {37F25},
	MRNUMBER = {4446269},
	DOI = {10.1007/s40598-020-00172-6},
}

@article {leau,
	AUTHOR = {Leau, L\'{e}opold},
	TITLE = {\'{E}tude sur les \'{e}quations fonctionnelles \`a une ou \`a plusieurs
	variables},
	JOURNAL = {Ann. Fac. Sci. Toulouse Sci. Math. Sci. Phys.},
	FJOURNAL = {Annales de la Facult\'{e} des Sciences de Toulouse pour les
	Sciences Math\'{e}matiques et les Sciences Physiques},
	VOLUME = {11},
	YEAR = {1897},
	NUMBER = {3},
	PAGES = {E25--E110},
	ISSN = {0996-0481},
	MRCLASS = {DML},
	MRNUMBER = {1508188},
	URL = {http://www.numdam.org/item?id=AFST_1897_1_11_3_E25_0},
}

@article {fatou_flower,
	AUTHOR = {Fatou, P.},
	TITLE = {Sur les \'{e}quations fonctionnelles},
	JOURNAL = {Bull. Soc. Math. France},
	FJOURNAL = {Bulletin de la Soci\'{e}t\'{e} Math\'{e}matique de France},
	VOLUME = {48},
	YEAR = {1920},
	PAGES = {208--314},
	ISSN = {0037-9484},
	MRCLASS = {DML},
	MRNUMBER = {1504797},
	URL = {http://www.numdam.org/item?id=BSMF_1920__48__208_1},
}

@book {lanford,
	AUTHOR = {Lanford, III, Oscar E. and Yampolsky, Michael},
	TITLE = {Fixed point of the parabolic renormalization operator},
	SERIES = {SpringerBriefs in Mathematics},
	PUBLISHER = {Springer, Cham},
	YEAR = {2014},
	PAGES = {viii+111},
	ISBN = {978-3-319-11706-5; 978-3-319-11707-2},
	MRCLASS = {37F25 (30D05 37F10)},
	MRNUMBER = {3308117},
	MRREVIEWER = {Artem Dudko},
	DOI = {10.1007/978-3-319-11707-2},
}

@misc{Yang,
	doi = {10.48550/ARXIV.1510.00043},
	
	url = {https://arxiv.org/abs/1510.00043},
	
	author = {Yang, Fei},
	
	title = {Parabolic and near-parabolic renormalizations for local degree three},
	
	publisher = {arXiv},
	
	year = {2015},
	
	copyright = {arXiv.org perpetual, non-exclusive license}
}
	\ \\
	\noindent\textsc{\small Department of Mathematics, Harvard University, 1 Oxford St, Cambridge, Massachusetts 02138}\par\nopagebreak
	\medskip
	\noindent\textit{E-mail address}: \href{mailto:kapiamba@math.harvard.edu}{\texttt{kapiamba@math.harvard.edu}}
		
\end{document}